\documentclass[11pt]{amsart}
\usepackage{amsmath, amsfonts, amssymb, amsbsy, chatterjee1, upref}

\begin{document}
\newcommand{\est}{e^{-t}}
\newcommand{\esst}{\sqrt{1-e^{-2t}}}
\newcommand{\bbu}{\mathbf{U}}
\newcommand{\bbv}{\mathbf{V}}
\newcommand{\boe}{{\boldsymbol \eta}}
\newcommand{\bos}{{\boldsymbol \sigma}}
\newcommand{\bu}{\mathbf{u}}
\newcommand{\bv}{\mathbf{v}}
\newcommand{\mc}{\mathcal{C}}
\newcommand{\om}{\overline{M}}
\newcommand{\oy}{\overline{Y}}
\newcommand{\oyp}{\overline{Y'}}
\newcommand{\oyy}{\overline{YY'}}
\newcommand{\ttb}{\mathbb{T}}

\title{Chaos, concentration, and multiple valleys}

\author{Sourav Chatterjee}
\thanks{Research partially supported by NSF grant DMS-0707054 and a Sloan Research Fellowship}
\address{\newline367 Evans Hall \#3860\newline
Department of Statistics\newline
University of California at Berkeley\newline
Berkeley, CA 94720-3860\newline
{\it E-mail: \tt sourav@stat.berkeley.edu}\newline 
{\it URL: \tt http://www.stat.berkeley.edu/$\sim$sourav}}


\begin{abstract}
Disordered systems are an important class of models in statistical mechanics, having the defining characteristic that the energy landscape is a fixed realization of a random field. Examples include various models of glasses and polymers. They also arise in other areas, like fitness models in evolutionary biology. The ground state of a disordered system is the state with minimum energy. The system is said to be chaotic if a small perturbation of the energy landscape causes a drastic shift of the ground state. We present a rigorous theory of chaos in disordered systems that confirms long-standing physics intuition about connections between chaos, anomalous fluctuations of the ground state energy, and the existence of multiple valleys in the energy landscape. Combining these results with mathematical tools like hypercontractivity, we establish the existence of the above phenomena 
in eigenvectors of GUE matrices, the Kauffman-Levin model of evolutionary biology, directed polymers in  random environment, a subclass of the generalized Sherrington-Kirkpatrick model of spin glasses, the discrete Gaussian free field, and continuous Gaussian fields on Euclidean spaces. 
We also list several open questions.
\end{abstract}

\maketitle

\setcounter{tocdepth}{1}
\tableofcontents

\section{Introduction}\label{intro}
Let us begin with a motivating example. Let $\bbg = (g_v)_{v\in \zz^2}$ be a collection of i.i.d.\ standard Gaussian random variables. A $(1+1)$-dimensional directed polymer of length $n$ is a sequence of $n$ adjacent points in $\zz^2$, beginning at the origin, such that each successive point is either to the right or above the previous point. The energy of a polymer $p = (v_0,\ldots,v_{n-1})$ in the Gaussian random environment $\bbg$ is defined as
\[
E(p) := -\sum_{i=0}^{n-1} g_{v_i}.
\] 
The `ground state' of the system is the polymer path with minimum energy, which we denote by $P$. One of the main goals of this paper is to understand a particular feature of the ground state, known as the   chaos property. It says, roughly, that a small perturbation of the environment gives rise to a new ground state that is almost disjoint from the original one.

There is a standard way to define a perturbation of a Gaussian environment. If $\bbg'$ is an independent copy of $\bbg$, and we define the perturbed environment $\bbg^t := \est \bbg + \esst \bbg'$, then $\bbg^t$ is again a standard Gaussian random environment. The parameter $t$ is a measure of the amount of perturbation. This definition arises naturally from running an Ornstein-Uhlenbeck flow at each vertex for time $t$. 

Formally, the property of chaos of the ground state means that there exists $t_0(n)$ such that $t_0(n) \ra 0$ and  $\sup_{t\ge t_0(n)} \ee|P\cap P^t| = o(n)$ as $n\ra \infty$, where $P^t$ is the minimum energy path in the environment $\bbg^t$, and $|P\cap P^t|$ is the number of vertices common to the two paths. The definition of `almost disjoint' in this way makes sense, because $P$ and $P^t$ are both of length $n$.

Although this is a widely studied phenomenon in the theoretical physics literature on directed polymers (see e.g.\ \cite{husehenleyfisher85}, \cite{zhang87}, \cite{mezard90}, \cite{fisherhuse91}, \cite{hwafisher94}, \cite{silveirabouchaud04}), there are no rigorous results. Using the techniques of this paper, we can prove the following theorem.
\begin{thm}\label{polyad}
Fix $n$, and let $t_0 = (\log n)^{-1/2}$. Then for all $t\ge t_0$,
\[
\ee|P\cap P^t|\le \frac{Cn}{\sqrt{\log n}},
\]
where $C$ is a universal constant.
\end{thm}
\noindent Of course, this result only proves chaos in principle. The factor of $\sqrt{\log n}$ is too slowly growing to be of any practical significance, even when $n$ is of the order of the Avogadro number.

The key to our approach is a seemingly new connection between the fluctuations of the ground state energy and the stability of the minimum energy path. Suppose $\tau$ is an Exponential random variable with mean $1$, independent of all else. Then we have the relation
\[
\var(\min_p E(p)) = \ee|P\cap P^\tau|.
\]
This comes as a consequence of a far more general result, that we are going to describe in the following pages. Given this formula, there are still two tasks remaining: (a) to show that the variance is $o(n)$, and (b) to prove a Tauberian theorem that extracts a bound on $\ee|P\cap P^t|$ for fixed $t$ from a bound on $\ee|P\cap P^\tau|$. We carry out task (a) in Section \ref{polymer} and task (b) in Section \ref{proof:main} (more specifically, via Theorem \ref{varconv}, which is one of the main results of this paper). 

Let us now present our general framework, that encompasses the polymer model as a special case. We treat the energy landscape of the model as a giant Gaussian random vector, and prove general theorems about Gaussian vectors that imply results like Theorem \ref{polyad}. There are many other examples that fall into this framework; besides polymers, the ones that have been treated in this paper include spin glass models, random matrices, fitness models of evolutionary biology, the discrete Gaussian free field, and continuous Gaussian fields on Euclidean spaces. 

Let $S$ be a finite set and $\bbx = (X_i)_{i\in S}$ be a centered Gaussian random vector with possibly dependent coordinates. (In the context of the polymer example, think of $S$ as the set of all directed polymers of length $n$ starting at the origin, and $X_i = -E(i)$ for a path $i\in S$.) Let
\[
R(i,j) := \cov(X_i,X_j), \ \ \ \sigma^2 := \max_{i\in S} \var(X_i).
\]
(Again, for polymers $R(i,j) = |i\cap j|$, and $\sigma^2 = n$.)
We will often refer to $\bbx$ as a `Gaussian field'. The elements of $S$ will be alternately called `states' or `indices' or `sites' or `coordinates'. 
The two central objects of interest in this paper are (i) the maximum of the Gaussian field $\bbx$,
\[
M := \max_{i\in S} X_i,
\]
and (ii) the location of the maximum,
\[
I := \argmax_{i\in S} X_i.
\]
To ensure that $I$ is well-defined, we assume the non-degeneracy condition
\begin{equation}\label{nondeg}
\pp(X_i \ne X_j)=1 \ \text{ for each $i\ne j$.}
\end{equation}
We study $M$ through its mean and variance
\[
m := \ee(M), \ \ \ v := \var(M).
\]
Let us alert the reader that we will use the symbols $\bbx$, $M$, $m$, $v$, $R(i,j)$, and $I$ throughout the paper to mean what they stand for here, often without explicit reference to the above definitions. 

The following is a well-known result about the fluctuations of Gaussian maxima.
\begin{prop}\label{varcrude}
Irrespective of the correlation structure of the vector $\bbx$, we always have $v \le \sigma^2$. 
\end{prop}
\noindent In words, this means that the order of fluctuations of the maximum cannot be larger than the order of the fluctuations of the most fluctuating coordinate. This inequality was proved by Houdr\'e~\cite{houdre95}, although the method of proof seems to be implicit in the much earlier work of Nash~\cite{nash58}, and the works of Chernoff~\cite{chernoff81}, Chen~\cite{chen82}, and  Houdr\'e and Kagan~\cite{houdrekagan95} on the so-called Poincar\'e inequality for the Gaussian measure.

There is a famous `advanced version' of the Poincar\'e inequality, called the Gaussian isoperimetric inequality, independently invented by Borell \cite{borell75} and Sudakov and Tsirelson \cite{sudakovtsirelson74}, that gives tail bounds instead of simply a variance inequality. A striking consequence of the isoperimetric inequality is the following result of Tsirelson, Ibragimov, and Sudakov~\cite{tis76}.
\begin{prop}\label{borell}
For any $r\ge 0$,
\[
\pp\bigl(M - m\ge r\bigr)\le e^{-r^2/2\sigma^2},
\]
and the same bound holds for $P(M-m \le -r)$ as well.
\end{prop}
\noindent Although the above result is often called `Borell's inequality', it is clearly not a fair nomenclature. Since it is too cumbersome to call it the `Borell-Tsirelson-Ibragimov-Sudakov inequality', we will simply refer to it as Proposition \ref{borell} in this manuscript. 

Note that although Proposition \ref{borell} is a deep and powerful result, conceptually it does not say a lot more than Proposition \ref{varcrude}, since it implies, just as Proposition \ref{varcrude},  that the fluctuations of the maximum can be at most of order $\sigma^2$. In particular, it is a crude worst case bound that does not use the correlation structure of~$\bbx$. However, this is all that one can obtain from the classical theory of concentration of measure (see e.g.\ Ledoux~\cite{ledoux01}).

Here is where our investigation begins. What happens if $v$ is very small compared to $\sigma^2$? As we will see, this is in fact the rule rather than the exception in interesting examples. The main point of this paper is that the condition $v\ll \sigma^2$ ushers in a whole host of interesting structure on the field~$\bbx$. Indeed, the structure is so interesting and pervasive that the condition  seems to deserve a name of its own. When it happens, we will say that the Gaussian field $\bbx$ exhibits `superconcentration'. The notion can be precisely defined only in terms of a sequence of Gaussian fields rather than a single one. Accordingly, let $(\bbx_n)_{n\ge 1}$ be a sequence of centered Gaussian fields, where $\bbx_n$ is defined on a finite set $S_n$. Let 
\[
\sigma_n^2 := \max_{i\in S_n} \var(X_{n,i}), \ \   M_n := \max_{i\in S_n} X_{n,i}, \ \ m_n := \ee(M_n), \ \ v_n = \var(M_n).
\]
\begin{defn}\label{superconc}
We say that the sequence of Gaussian fields $(\bbx_n)_{n\ge 1}$ `has superconcentrated maximum' or simply `is superconcentrated' if $v_n = o(\sigma_n^2)$ as $n \ra \infty$. 
\end{defn}
\noindent Here, as usual, $a_n = o(b_n)$ means that $\lim_{n\ra\infty} a_n/b_n = 0$. In practice, we will simply say that $\bbx$ is superconcentrated if it is implicitly the $n$th member of a sequence of fields having superconcentrated maximum. We will give many examples of superconcentrated Gaussian fields in the subsequent sections to demonstrate the `rule rather than exception' claim.

Incidentally, physicists often refer to the superconcentration phenomenon as `anomalous fluctuations' (see e.g.\ \cite{hwafisher94}). However, it is not a well-defined notion (in particular, they don't connect it with classical concentration, and `anomalous fluctuations' can also mean larger fluctuations than usual); we feel that our terminology is more evocative and precise.

Let us now describe some of the consequences of superconcentration. A summary of the results is contained in Theorem \ref{main}, but we first need to define some concepts. 

An important property of Gaussian fields that has been studied in various special examples by physicists but has almost no presence in rigorous mathematics, is the property of chaos. We have already had a discussion of this in the context of polymers, so let us now make a general definition. Let $\bbx'$ be an independent copy of $\bbx$. For each $t\in [0,\infty)$, let 
\[
\bbx^t := \est \bbx + \esst \bbx'.
\] 
Note that $\bbx^t$ has the same distribution as $\bbx$, that is, the transformation $\bbx \ra \bbx^t$ is a distribution preserving perturbation of $\bbx$.   We will be mostly interested in small perturbations, i.e., small $t$. As mentioned before, this is a natural way to define perturbations of Gaussian fields because of its intimate relation to Ornstein-Uhlenbeck diffusions. Let $I^t$ be the state at which the maximum is attained in $\bbx^t$, that is,
\[
I^t := \argmax_{i\in S} X^t_i.
\]
Note that $I^t$ is well-defined by assumption \eqref{nondeg}, and that $I^0 = I$. We will say that the field $\bbx$ is `chaotic' if $I^t$ is highly unstable, that is a small change in the value of $t$  causes, with high probability, a drastic change in $I^t$. There are, of course, various notions to be made precise here. First of all, what is meant by a drastic change in $I^t$?  As we will see, in most examples two states $i,j\in S$ are `drastically different' if $R(i,j)$ is very small, typically $R(i,j)\ll \sigma^2$. Thus, we may formulate chaos for $I^t$ to mean that for $t = o(1)$, $\ee( R(I^0,I^t)) \ll \sigma^2$. Secondly, what is a `small change in $t$'? This we will take at face value, i.e.\ small means small, in relation to nothing else.

However, none of this is precise. As before, the only way to make a completely meaningful definition is via sequences. 

Accordingly, let $(\bbx_n)_{n\ge 1}$ be a sequence of Gaussian fields as in Definition~\ref{superconc}. Let $\bbx_n'$ be an independent copy of $\bbx_n$ and define, as above, the perturbed fields
\[
\bbx_n^t := \est \bbx_n + \esst \bbx_n'.
\]
Again, as above, let $I^t_n = \argmax_{i\in S_n} X_{n,i}^t$. Let $R_n$ be the covariance kernel of $\bbx_n$. We are now ready to give a precise definition of the chaos property.
\begin{defn}
We say that the location of the maximum in $\bbx_n$ exhibits `chaos' or simply that $\bbx_n$ exhibits chaos if there is a sequence $t_n \ra 0$ such that $\ee R_n(I_n^0, I_n^{t_n}) = o(\sigma_n^2)$ as $n \ra \infty$. 
\end{defn}
\noindent Note that here we defined chaos in terms of the decay of $\ee R_n(I_n^0, I_n^{t_n})$, instead of $\sup_{t\ge t_n} \ee R_n(I_n^0, I_n^t)$ as we did for the polymer example. It turns out (Theorem \ref{varconv}) that $\ee R_n(I_n^0, I_n^t)$ is always a decreasing function of $t$, and therefore the two definitions are equivalent.

Next let us turn our attention to the so-called `multiple valley picture'. Often, we have the situation that a Gaussian field $\bbx$ has many `drastically different' sites at which the global maximum is nearly achieved. It is called `multiple valleys' instead of `multiple peaks' because the physicists like to put a minus sign. We will, however, call it the multiple peaks phenomenon to avoid any confusion. As before, we attempt to give a precise definition via sequences of fields. All notation is the same as before.
\begin{defn}
We say that $\bbx_n$ has multiple peaks (MP) if there exist $l_n \ra \infty$, $\epsilon_n = o(\sigma_n^2)$, $\delta_n = o(m_n)$ and $\gamma_n \ra 0$ such that with probability at least $1- \gamma_n$, there is a set $A \subseteq S_n$ satisfying
\begin{enumerate}
\item[(a)] $|A|= l_n$,
\item[(b)] $R_n(i,j) < \epsilon_n$ for each $i,j\in A$, $i\ne j$, and 
\item[(c)] $X_i \ge M_n - \delta_n$ for each $i\in A$. 
\end{enumerate}
\end{defn}
\noindent Note that the condition $\delta_n = o(m_n)$ is natural, because $X_i$ is nearly maximum if $X_i = M_n - o(M_n)$. However, this is not the form that is conjectured in the physics models. Indeed, since the fluctuations of $M_n$ are of order $\sqrt{v_n}$, the physicists seem to think that multiple peaks, in the above sense, should occur whenever $\delta_n\gg \sqrt{v_n}$, or at least when $\delta_n = o(\sigma_n)$. This leads us to the definition of a stronger notion of multiple peaks.
\begin{defn}
We say that $\bbx_n$ has strong multiple peaks if the multiple peaks condition is satisfied with $\delta_n = o(\sigma_n)$ instead of $\delta_n= o(m_n)$. 
\end{defn}
\noindent The first main result of this paper, stated below, shows that the properties of superconcentration, chaos, multiple peaks, and strong multiple peaks are all intimately related to each other. This may not be surprising from a physicist's point of view, but this is the first time that these widely observed phenomena have been formulated and connected by rigorous mathematics.
\begin{thm}\label{main}
For any sequence of Gaussian fields $(\bbx_n)_{n\ge 1}$ satisfying the non-degeneracy condition \eqref{nondeg} we have
\[
\text{Strong MP} 
\begin{array}{c}
\Longrightarrow\\
\not \Longleftarrow 
\end{array}
\text{Superconcentration} \iff \text{Chaos.} 
\]
Moreover, under the `positivity assumption' that $R_n(i,j)\ge 0$ for each $n$ and $i,j \in S_n$, we have the more complete picture:
\[
\text{Strong MP} 
\begin{array}{c}
\Longrightarrow\\
\not \Longleftarrow 
\end{array}
\text{Superconcentration} \iff \text{Chaos} 
\begin{array}{c}
\Longrightarrow\\
\not \Longleftarrow 
\end{array}
\text{ MP.}
\]
\end{thm}
\noindent The counterexample that shows chaos $\not \Rightarrow$ strong MP is particularly surprising and goes against common intuition. It shows that, contrary to what one may think, chaos {\it is not necessarily caused} by the existence of multiple peaks.

Theorem \ref{main} is stated as a limiting result, but we do have precise quantitative bounds for all parts of the theorem. These will be presented in Section~\ref{proof:main}, where we give the proof of Theorem~\ref{main}. One result from Section~\ref{proof:main} is worth mentioning here, since it provides the foundation for all subsequent work by connecting the value of the maximum with the location of the maximum in a Gaussian field. It is Lemma \ref{varmax}, which says that if $\tau$ is an Exponential random variable with mean $1$, independent of all else, then we have the formula
\[
v = \ee(R(I^0,I^\tau)).
\]
We have already stated the version of this formula for polymers. The proof of this identity would be easy for any expert on the Ornstein-Uhlenbeck semigroup --- indeed, it is implicit in the classical proofs of Propositions~\ref{varcrude} and \ref{borell} --- but the author has not seen it explicitly written down anywhere in the published literature. Interestingly, it was brought to our notice that the identity does make an appearance (in a slightly different form) in a recent manuscript of Nourdin and Viens~\cite{nourdinviens08} that was prepared at the same time as the first version of this paper was being written.

In spite of its cuteness the identity is not very useful on its own, since we are interested in $\ee(R(I^0, I^t))$ for fixed~$t$. This question is handled in Theorem~\ref{varconv}, where we apply a Tauberian argument to the above representation of $v$ to show that for each~$t$, 
\begin{align*}
&0\le \ee(R(I^0,I^t)) \le \frac{v}{1-e^{-t}},
\ \text{ and } \\
&v \le \sigma^2(1-e^{-t}) + \ee(R(I^0,I^t)) e^{-t}.
\end{align*}
It turns out that these upper and lower bounds suffice to show the equivalence of superconcentration and~chaos (we show this in Section \ref{proof:main}). 

Before presenting further results, let us discuss some of the literature. As we mentioned before, the phenomena of superconcentration, chaos, and multiple peaks have not been systematically studied in the mathematics literature, so there are essentially very few references. The closest thing to chaos in the domain of rigorous mathematics is the notion of noise-sensitivity, although that refers mainly to correlations between functions. The literature on noise-sensitivity in computer science is sizable; it mostly involves sensitivity of scalar functions of Boolean random variables to random noise, which is not so relevant to us. In the probability world, a very notable paper on the subject is due to Benjamini, Kalai, and Schramm~\cite{bks01}. A subsequent paper~\cite{bks03} by the same authors is more relevant for what we do in this article. 

One truly significant contribution to what we call superconcentration is due to Talagrand (\cite{talagrand94}, Theorem 1.5), which has been the source of many applications (including \cite{bks03}). Talagrand's result provides a way to improve variance bounds like Proposition \ref{varcrude} under certain situations by a `factor of $\log n$'. Talagrand's breakthrough idea was to use the tool of hypercontractivity (discovered by Nelson \cite{nelson73}) to improve variance bounds. We will use this method, in conjunction with ideas from Benjamini-Kalai-Schramm~\cite{bks03} and our Theorem \ref{main}, to prove the existence of chaos in directed polymers in Section \ref{polymer}. 

Another contribution of Talagrand in our context is the proof of strong MP $\Rightarrow$ superconcentration, which follows essentially from a sketch at the end of Section 8.3 in his landmark paper \cite{talagrand95}. Unfortunately, the author has not yet been able to find a use for this remarkable implication, since proving strong MP seems to be always more difficult than proving any of the other phenomena.

In the physics literature, there is a long, folklorish history of studying the phenomena of superconcentration (via `fluctuation exponents'), chaos, and multiple valleys. For instance, the implication that superconcentration $\Rightarrow$ chaos is the central theme of  Fisher and Huse \cite{fisherhuse91}, who investigated it in the context of directed polymers in a random environment. Examples of other highly cited works in this area are those of McKay, Berger, and Kirkpatrick~\cite{mckayetal82}, Huse, Henley, and Fisher \cite{husehenleyfisher85},  Bray and Moore \cite{braymoore87}, Zhang~\cite{zhang87} and M\'ezard~\cite{mezard90}.

Let us now return to the discussion of our results. As we mentioned above, the only rigorous tool available at present that can establish superconcentration is Talagrand's method of using hypercontractivity to improve variance bounds. Although this works in many situations (including some of our examples in this paper), it gives only a `$\log n$ correction', and usually does not suffice to break the barrier of `improvements in powers of $n$'. 

One of the main goals of this work is to break this wall by finding an alternative technique. We have only had partial success in this direction, but what we have may lead to further progress. Under a certain condition that we call `extremality', we are able to get improvements in powers of $n$ in highly nontrivial models like certain cases of the generalized Sherrington-Kirkpatrick model of spin glasses. The success is `partial' because, for instance, we are not able to cover the original $2$-spin SK model.

The notion of extremality of a Gaussian field like $\bbx$ is defined as follows. It can be shown (see Lemma \ref{max}) that irrespective of the correlation structure,
we have $m \le \sqrt{2\sigma^2\log |S|}$,
with near equality if the coordinates are i.i.d.\ $N(0,\sigma^2)$. We say that the field $\bbx$ is `extremal' if $m \simeq \sqrt{2\sigma^2\log |S|}$. Of course this makes sense only if we consider sequences.
\begin{defn}\label{extremaldef}
We say that $(\bbx_n)_{n\ge 1}$ is extremal if 
\[
\lim_{n \ra \infty} \frac{m_n}{\sqrt{2\sigma_n^2\log |S_n|}} = 1.
\]
\end{defn}
\noindent Note that although extremality may seem to indicate that the coordinates of $\bbx_n$ are approximately independent for large $n$, that is not true. Extremality can hold even in models with high degrees of dependence between sites, like the Gaussian free field (proved by Bolthausen, Deuschel, and Giacomin~\cite{bdg01}), branching random walks (proved by Biggins \cite{biggins77}), and some nontrivial spin glasses (treated in Section \ref{skmodel}). The second main result of this paper, proved in Section \ref{proof:main2}, is the following. 
\begin{thm}\label{main2}
For any sequence of Gaussian fields $(\bbx_n)_{n\ge 1}$ with state spaces $S_n$ growing in size to infinity, we have
\[
Extremality \Longrightarrow Chaos.
\]
\end{thm}
\noindent Let us now present one example of a concrete variance bound that can be used to establish superconcentration. The following theorem is proved in Section \ref{proof:main2}, where it is deduced from a quantitative version of Theorem~\ref{main2}. We shall use this result to establish the presence of chaos in certain models of spin glasses in Section \ref{skmodel}. Moreover, the variance bound given by the following theorem, wherever it applies, gives `corrections in powers of $n$' instead of log corrections as given by hypercontractivity. 
\begin{thm}\label{gencorr}
Consider the field $\bbx = (X_i)_{i\in S}$. Suppose $R(i,i) = \sigma^2$ for all $i$. For each $i,j\in S$, let $r_{ij} := R(i,j)/\sigma^2$. Let
\[
\beta :=  \biggl(\frac{\log \log |S| +\log \sum_{i,j\in S} |S|^{-2/(1+r_{ij})}}{\log |S|}\biggr)^{1/4}. 
\]
Then $v\le C\sigma^2\beta$ and for any $t\ge 0$, $\ee(R(I^0,I^t)) \le \frac{C\sigma^2\beta}{\esst}$, where $C$ is a universal constant.
\end{thm}
\noindent The bound can be interpreted easily by considering i.i.d.\ standard Gaussians. If $r_{ij}=0$ for all  $i\ne j$ and $r_{ii}=1$ for all $i$, we have 
$\sum_{i,j\in S} |S|^{-2/(1+r_{ij})} \le 2$, 
which proves superconcentration (although not the correct order bound on the variance in this case). In general, Theorem \ref{gencorr} gives a way of proving superconcentration and chaos when `most correlations are small', but may not work in many situations.
 
The rest of the paper is organized as follows. In Section \ref{basicfacts}, we state some well-known  results about Gaussian random variables that will be useful for us later on. In Section \ref{proof:main}, we prove Theorem \ref{main}, together with a number of results that give quantitative versions of the various implications in Theorem~\ref{main}. In Section \ref{hypercon}, we give a brief introduction to the concept of hypercontractivity for the Ornstein-Uhlenbeck semigroup and how to use it for proving superconcentration. In Section \ref{proof:main2} we prove Theorems~\ref{main2} and \ref{gencorr}. Finally, in Sections~\ref{eigen} through \ref{euclidean}, we work out a number of examples. This includes applications to eigenvectors of random matrices (Section \ref{eigen}), the Kauffman-Levin $NK$ fitness model of evolutionary biology (Section \ref{nk}), directed polymers in random environment (Section \ref{polymer}), the generalized Sherrington-Kirkpatrick model of spin glasses (Section \ref{skmodel}), the discrete Gaussian free field (Section \ref{dgff}), and Gaussian fields on Euclidean spaces (Section \ref{euclidean}). 

Let us now mention some conventions that we will follow in this paper. First of all, we must declare  that the constant $C$ is going to stand for any generic universal constant, whose value may change from line to line. This is an invaluable help in lightening notation. We will generally denote scalar variables and elements of $\rr^2$ by ordinary italic font variable names like $x,y,u,$ etc. We will use boldface in dimensions higher than $2$. Finally, let us reiterate that the symbols $\bbx,\bbx^t, M, m, v,\sigma^2, I^t$, and $R(i,j)$ will be used without reference to denote what they denote in this section.

\section{Some basic facts about Gaussian random variables}\label{basicfacts}
In this section, as elsewhere, we continue to use the notation defined in Section \ref{intro}. In particular, $\bbx$, $\bbx^t$, $m$, $v$, $\sigma^2$, $R(i,j)$, and $I^t$ stand for the same objects as before. We state some very well-known facts about Gaussian random variables and vectors that will be of  repeated use for us in the rest of the manuscript. Two such facts, namely Proposition \ref{varcrude} and Proposition \ref{borell}, have already been stated in Section \ref{intro}. 
\vskip.1in
\noindent\textbf{Size  of the maximum.} Just like the variance, the expected value of the maximum of a Gaussian field also has a general, worst case bound.  The bound is much easier to establish than the variance bound, so we give the proof right here. 
\begin{lmm}\label{max}
We have the general bound
\[
m\le \sqrt{2\sigma^2\log |S|}.
\]
Moreover if $|S|\ge 2$ then for any $p\ge 1$, 
\[
\ee|M|^p\le\ee\max_i |X_i|^p\le  C(p)\sigma^p (\log |S|)^{p/2},
\]
where $C(p)$ is a constant that depends only on $p$. 
\end{lmm}
\noindent{\it Remark.} We do not need that $(X_i)_{i\in S}$ are Gaussian for the bound on the expectation; the proof goes through for any collection random variables with Gaussian tails, irrespective of the dependence among them. This observation will be used a few times in the sequel.
\begin{proof}
Without loss of generality, assume that $\sigma^2 = 1$. Then for any $\beta > 0$,
\begin{align*}
m &= \frac{1}{\beta}\ee(\log e^{\beta M}) \\
&\le \frac{1}{\beta}\ee\biggl(\log \sum_{i\in S} e^{\beta X_i}\biggr)\\
&\le \frac{1}{\beta}\log \sum_{i\in S} \ee(e^{\beta X_i}) \le \frac{\beta}{2} + \frac{\log|S|}{\beta}. 
\end{align*}
Optimizing over $\beta$, we establish the first claim. For the second, one just has to combine the bound on $m$ with Proposition \ref{borell}, and observe that $\max |X_i|$ is the maximum of the concatenation of the vectors $\bbx$ and $-\bbx$. 
\end{proof}
\noindent {\bf Slepian's lemma and Sudakov minoration.} The following result is an indispensable tool in the study of Gaussian processes. It was discovered by Slepian \cite{slepian62} and goes by the name of `Slepian's lemma'.
\begin{lmm}\label{slepian}
Suppose $\bbx = (X_i)_{i\in S}$ and $\bby = (Y_i)_{i\in S}$ are centered Gaussian random vectors with $\ee(X_i^2)= \ee(Y_i^2)$ for each $i$ and $\ee(X_i X_j) \ge \ee(Y_i Y_j)$ for each $i,j$. Then for each $x\in \rr$,
\[
\pp(\max_i X_i > x) \le \pp(\max_i Y_i > x).
\]
In particular, $\ee(\max_i X_i) \le \ee(\max_i Y_i)$. 
\end{lmm}
\noindent The next result is a close analog of Slepian's inequality, known as the Sudakov minoration lemma. For a proof, see Lemma 2.1.2 in \cite{talagrand05}.
\begin{lmm}
Suppose $a$ is a constant such that $\ee(X_i-X_j)^2 \ge a$ for all $i,j\in S$, $i\ne j$. Then
\[
m \ge C a\sqrt{\log |S|},
\]
where $C$ is a positive universal constant. 
\end{lmm}
\vskip.1in
\noindent{\bf Mills ratio bounds.}
The following pair of inequalities is collectively  known as the Mills ratio bounds. For a standard Gaussian random variable $Z$, for any $x>0$, we have
\begin{align*}
\frac{xe^{-x^2/2}}{\sqrt{2\pi}(1+x^2)}\leq \pp(Z > x) \leq \frac{e^{-x^2/2}}{x\sqrt{2\pi}}.\label{MillsRatio}
\end{align*}
The proof is not difficult, and may be found in numerous standard texts on probability and statistics. The inequalities in the above form were probably first proven by Gordon \cite{gordon41}. Along similar lines, one can also prove the inequality
\begin{equation}\label{gausstail}
\pp(Z > x) \le e^{-x^2/2},
\end{equation}
which follows simply by optimizing over $\pp(Z> x) \le e^{-\theta x} \ee(e^{\theta Z})$. 
\vskip.1in
\noindent {\bf Gaussian integration by parts.} Suppose $Z$ is a standard Gaussian random variable, and $f:\rr \ra \rr$ is an absolutely continuous function. If $\ee|f'(Z)| <\infty$, then one can argue that 
\begin{equation}\label{zfz}
\ee|Zf(Z)|\le C \ee|f'(Z)| + C < \infty,
\end{equation}
where $C$ is a universal constant. Moreover, a standard application of integration by parts gives the well-known identity
\[
\ee(Zf(Z)) = \ee (f'(Z)).
\]
This identity can be easily generalized to a Gaussian random vector like $\bbx$, as follows. If $f:\rr^S \ra \rr$ is an absolutely continuous function such that $\|\nabla f(\bbx)\|$ has finite expectation, then for any $i\in S$,
\[
\ee(X_i f(\bbx)) = \sum_{j\in S} R(i,j) \ee(\partial_j f(\bbx)),
\]
where $\partial_j f$ denotes the partial derivative of $f$ along the $j$th coordinate. This identity can be derived from the previous one simply by writing $\bbx$ as a linear transformation of a vector of i.i.d.\ standard Gaussian random variables. The author encountered this useful version of the integration-by-parts identity in \cite{talagrand03}, Appendix A.6.

\section{Structure of a superconcentrated Gaussian field}\label{proof:main}
The goal of this section is to prove Theorem \ref{main}. 
Throughout this section, as everywhere else, we will freely use notation from Section \ref{intro} (like $\bbx^t$, $R(i,j)$, $m$, $v$, $I^t$, and $M$) without explicit reference. 
We will divide the proof into a number of subsections, one devoted to each part of the proof. The theorems of this section are all interesting in their own right, because they give quantitative versions of the various implications of Theorem \ref{main}. 
\subsection{Superconcentration is equivalent to Chaos}
We begin with the exact formula for $v$ stated in Section \ref{intro}. 
\begin{lmm}\label{varmax}
Let $\tau$ be a standard  exponential random variable, independent of everything else. Then we have
\begin{equation}\label{mainform}
v = \ee(R(I^0,I^\tau)).
\end{equation}
\end{lmm}
\noindent Note that by the Cauchy-Schwarz inequality, this identity implies Proposition \ref{varcrude}.

Lemma \ref{varmax}, combined with an elementary  Tauberian argument, leads to the following Theorem, which establishes the equivalence of superconcentration and chaos.

\begin{thm}\label{varconv}
For each $t$ we have
\begin{align*}
&0\le \ee(R(I^0,I^t)) \le \frac{v}{1-e^{-t}},
\ \text{ and } \\
&v \le \sigma^2(1-e^{-t}) + \ee(R(I^0,I^t)) e^{-t}.
\end{align*}
Moreover, $\ee(R(I^0,I^t))$ is a decreasing function of $t$.
\end{thm}
\noindent To see how the bounds imply the claim of equivalence of superconcentration and chaos, consider the following. If, in the notation of Section \ref{intro}, we have $v_n= o(\sigma_n^2)$, then by choosing $t_n = \sqrt{v_n/\sigma_n^2} = o(1)$ we can guarantee by the first bound that 
\[
\ee(R(I_n^0, I_n^{t_n})) \le O(\sqrt{v_n} \sigma_n) = o(\sigma_n^2).
\]
Again, if we can find $t_n = o(1)$ such that $\ee(R(I_n^0, I_n^{t_n})) = o(\sigma_n^2)$, then the second bound shows that $v_n = o(\sigma_n^2)$.

The proof of Lemma \ref{varmax} is done in three simple steps.
\begin{lmm}\label{intparts}
Let $\bby = (Y_i)_{i\in S}$ be a vector of independent standard Gaussian random variables, and let $\bby'$ be an independent copy of $\bby$. Let $f:\rr^S\ra \rr$ be an absolutely continuous with gradient $\nabla f$ and suppose that $\ee\|\nabla f(\bby)\|^2<\infty$. For each $t\in [0,\infty)$, let $\bby^t = \est \bby + \esst \bby'$. Then
\begin{equation}\label{firsteq}
\var (f(\bby)) = \int_0^\infty e^{-t} \ee\avg{\nabla f(\bby), \nabla f(\bby^t)} \; dt,
\end{equation}
where $\smallavg{\cdot, \cdot}$ denotes the usual inner product.
\end{lmm}
\begin{proof}
Note that
\begin{align*}
&\var(f(\bby) ) = \ee(f(\bby)(f(\bby) - f(\bby'))) \\
&= \ee \biggl(- \int_0^\infty f(\bby)\frac{d}{dt} f(\est \bby + \esst \bby') dt\biggr)\\
&= \ee\biggl(-\int_0^\infty f(\bby) \sum_{i\in S} \biggl(-e^{-t}Y_i + \frac{e^{-2t}Y'_i}{\esst}\biggr)\partial_i f(\est \bby + \esst \bby') dt\biggr).
\end{align*}
Now fix $t\in [0,\infty)$, and let
\[
\bbv^t:=\esst \bby - \est \bby'.
\]
Then $\bby^t$ and $\bbv^t$ are independent standard Gaussian random vectors and 
\[
\bby = \est \bby^t + \esst \bbv^t.
\]
Taking any $i$, and using Gaussian integration-by-parts as outlined in Section~\ref{basicfacts} (in going from the second to the third line below), we get
\begin{align*}
&\ee\biggl(f(\bby) \biggl(\est Y_i - \frac{e^{-2t}Y'_i}{\esst}\biggr)\partial_i f (\est \bby + \esst \bby')\biggr) \\
&= \frac{\est}{\esst}\ee\biggl(f(\est \bby^t + \esst \bbv^t) V^t_i \partial_i f (\bby^t)\biggr)\\
&= e^{-t} \ee\bigl( \partial_i f (\bby) \partial_i f (\bby^t)\bigr).
\end{align*}
(Note that the condition $\ee\|\nabla f(\bby)\|^2< \infty$ and the bound \eqref{zfz} allows us to integrate  by parts and interchange integrals and expectations.) 
This completes the proof of the Lemma.
\end{proof}

\begin{lmm}\label{intparts2}
Let $f$ be as in Lemma \ref{intparts}. Then
\[
\var (f(\bbx)) = \int_0^\infty e^{-t} \sum_{i, j =1}^n R(i,j) \ee(\partial_i f(\bbx) \partial_j f(\bbx^t)) \; dt.
\]
\end{lmm}
\begin{proof}
If $R = BB^T$ for some matrix $B$, then we can assume $\bbx = B \bby$ where $\bby$ is a standard Gaussian r.v. As before, let $\bby'$ be an independent copy of $\bby$ so that $\bbx' = B\bby'$ is an independent copy of $\bbx$. Putting $g(\by) := f (B \by)$, and using Lemma \ref{intparts} with $g$ instead of $f$, an easy computation gives
\begin{equation}\label{indep}
\sum_{i=1}^n \partial_i g(\bby) \partial_i g(\bby^t) = \sum_{i, j =1}^n R(i,j) \partial_i f(\bbx) \partial_j f(\bbx^t).
\end{equation}
By Lemma \ref{intparts}, this completes the proof.
\end{proof}

\begin{proof}[Proof of Lemma \ref{varmax}]
Consider the function $f(\bx) := \max_{i\in S} x_i$. We have
\[
\partial_i f(\bx ) = 1_{\{x_i \ge  x_j \forall j\}} \ \ \text{a.e.}
\]
The proof now follows easily from Lemma \ref{intparts2}.
\end{proof}

The proof of Theorem \ref{varconv} requires one more lemma. The current proof of the following result is a major simplification (thanks to Michel Ledoux) of the author's proof in the first draft. 
\begin{lmm}\label{ledoux}
Let $\bby$ and $\bby^t$ be as in Lemma \ref{intparts}. 
Then for any function $h:\rr^n \ra \rr$ such that $\ee(h(\bby)^2) < \infty$, $\ee(h(\bby)h(\bby^t))$ is a nonnegative, non-increasing function of $t$.
\end{lmm}
\begin{proof}
Fix $t\ge 0$. Let $\bby''$ be another independent copy of $\bby$. Let
\begin{align*}
\bby^{-t/2} &:= e^{-t/2} \bby + \sqrt{1-e^{-t}} \bby''.
\end{align*}
It is easy to verify (by checking covariances) that the pair $(\bby, \bby^t)$ has the same distribution as the pair $(\bby^{-t/2},\bby^{t/2})$. Again, it is trivial to see that given $\bby$, the vectors $\bby^{-t/2}$ and $\bby^{t/2}$ are independent and identically distributed. Thus,
\begin{equation}\label{ledouxrep}
\begin{split}
\ee(h(\bby) h(\bby^t)) &= \ee(h(\bby^{-t/2})h(\bby^{t/2}))\\
&= \ee\bigl((P_{t/2} h(\bby))^2\bigr),
\end{split}
\end{equation}
where
\[
P_s h(\bby) := \ee(h(\bby^s)\mid \bby). 
\]
This shows that $\ee(h(\bby)h(\bby^t))$ is nonnegative. Now, it is easy to verify that $(P_s)_{s\ge 0}$ is a semigroup of operators, that is, $P_s P_t = P_{s+t}$. Using this, note that for any $s,t\ge 0$, 
\begin{align*}
\ee\bigl((P_t h(\bby))^2 - (P_{t+s}h(\bby))^2\bigr)&= \ee\bigl((P_t h(\bby^s))^2 - (\ee(P_th(\bby^s)\mid \bby))^2\bigr)\\
&= \ee\bigl(\var(P_t h(\bby^s) \mid \bby)\bigr)\ge 0. 
\end{align*} 
Combined with the representation \eqref{ledouxrep}, this shows that $\ee(h(\bby)h(\bby^t))$ is non-increasing in $t$.
\end{proof}
\begin{proof}[Proof of Theorem \ref{varconv}]
By Lemma \ref{ledoux} and the representation \eqref{indep}, we see that $\ee(R(I^0,I^t))$ is nonnegative and non-increasing as a function of $t$. Combining this with Lemma  \ref{varmax}, we get that for any $t$,
\begin{align*}
v &= \int_0^\infty e^{-s}\ee(R(I^0,I^s)) ds\\
&\ge \int_0^te^{-s} \ee(R(I^0,I^s)) ds\\
&\ge \int_0^t e^{-s} \ee(R(I^0, I^t)) ds = (1-e^{-t}) \ee(R(I^0,I^t)).
\end{align*}
Similarly,
\begin{align*}
v &\le \sigma^2 \int_0^t e^{-s} ds + \ee(R(I^0,I^t)) \int_t^\infty e^{-s} ds\\
&= \sigma^2 (1-e^{-t}) + \ee(R(I^0,I^t)) e^{-t}.
\end{align*}
This completes the proof.
\end{proof}

\subsection{Strong Multiple Peaks implies Superconcentration}
As usual, we give a quantitative version of this result. The difference with other parts of the proof is that this proof is not an original idea of the author; it follows from a sketch given at the end of Section 8.3 in Talagrand's famous treatise on concentration inequalities  \cite{talagrand95}. To the best of our knowledge, this sketch has never been formulated as a concrete theorem before. 
\begin{thm}
Suppose $l$ is a positive integer, and $\epsilon > 0$, $\delta > 0$, and $\gamma\in [0,1]$ are numbers such that with probability $1-\gamma$, there exists a set $A\subseteq S$ such that 
\begin{enumerate}
\item[(a)] $|A|\ge l$,
\item[(b)] $R(i,j)<  \epsilon$ for all $i,j\in A$, $i\ne j$, and 
\item[(c)] $X_i \ge M - \delta$ for all $i\in A$.
\end{enumerate}
Then $v\le 3\delta^2 + 3l^{-1}\sigma^2 + 3\epsilon + 16\sigma^2\sqrt{\gamma}$.
\end{thm}
\noindent It is easy to see from the above result how we get superconcentration if we have a sequence of Gaussian fields as in Section \ref{intro} satisfying the strong MP condition with $l_n \ra \infty$, $\epsilon_n = o(\sigma_n^2)$, $\delta_n = o(\sigma_n)$ and $\gamma_n \ra 0$. 
\begin{proof}
Let 
\[
U := \{(i_1,i_2,\ldots,i_l): i_1,\ldots,i_l\in S, \; R(i_p, i_q) < \epsilon \text{ for all } p\ne q\}.
\]
For each $(i_1,\ldots,i_l)\in U$, let
\[
Z_{(i_1,\ldots,i_l)} := \frac{1}{l}\sum_{p=1}^l X_{i_p}.
\]
A simple computation shows that 
\[
\var(Z_{(i_1,\ldots,i_l)}) \le l^{-1}\sigma^2 + \epsilon. 
\]
Thus, if we let
\[
\om := \max_{(i_1,\ldots,i_l)\in U} Z_{(i_1,\ldots,i_l)}, 
\]
then by Proposition \ref{varcrude}
\begin{equation}\label{vam}
\var(\om) \le l^{-1}\sigma^2 + \epsilon.
\end{equation}
Now, let $\bbx'$ be an independent copy of $\bbx$, and define $M'$,  $Z'_{(i_1,\ldots,i_l)}$, and $\om'$ accordingly. 
Let $E$ denote the event of the existence of a set $A$ satisfying (a), (b), and (c) for the vector $\bbx$ and a set $A'$ satisfying (a), (b), and (c) for the vector $\bbx'$. If $E$ happens, then $|M-\om|\le \delta$ and $|M'-\om'|\le \delta$. By the inequality $(x+y+z)^2 \le 3(x^2 + y^2 + z^2)$ and the Cauchy-Schwarz inequality in the third step below, we have
\begin{align*}
2\var(M) &= \ee(M - M')^2\\
&= \ee((M - M')^2; E) + \ee((M-M')^2; E^c)\\
&\le 6\ee((M-\om)^2; E) + 3\ee((\om - \om')^2; E) \\
&\quad + \bigl[\pp(E^c)\ee(M-M')^4 \bigr]^{1/2}\\
&\le 6\delta^2 + 3\ee(\om-\om')^2 + (2\gamma\;\ee(M-M')^4)^{1/2}.
\end{align*}
By \eqref{vam}, the second term is bounded by $6(l^{-1}\sigma^2+\epsilon)$. By Proposition \ref{borell},
\begin{align*}
\ee(M-M')^4 &\le 16\sigma^4\int_0^\infty x^3 e^{-x^2/8} dx = 512\sigma^4.
\end{align*}
This completes the proof.
\end{proof}

\subsection{Under positivity, Chaos implies Multiple Peaks} The goal of this subsection is to prove that if $R(i,j) \ge 0$ for all $i, j$, then the property of chaos guarantees the multiple peaks condition. As before, we have a quantitative statement. 
\begin{thm}\label{p2p3}
Suppose $R(i,j)\ge 0$ for all $i,j$.  Then for any integer $l\ge 2$, any $\epsilon \in (0,\sigma^2)$, and any $\delta \in (0,m)$, we have that
with probability at least 
\begin{equation}\label{prob}
1 - 4\sqrt{\frac{vml^3\log l}{\delta \epsilon }}  - 4\biggl(\frac{v\sigma^4l^5 \log l}{\delta^3 m \epsilon}\biggr)^{1/4}
\end{equation}
there exists $A\subseteq S$ such that
\begin{enumerate}
\item[(a)] $|A| = l$,
\item[(b)] $R(i,j) < \epsilon$ for all $i,j\in A$, $i\ne j$, and
\item[(c)] $X_i\ge M -  \delta$ for all $i\in A$.
\end{enumerate}
\end{thm}
\noindent It is not immediately clear why this theorem shows that chaos (or equivalently, superconcentration) implies multiple peaks. Let us prove that implication before starting the proof of Theorem \ref{p2p3}. Recall the sequence $\bbx_n$ from Section \ref{intro} and the associated quantities. Suppose $v_n = o(\sigma_n^2)$. Let
\[
\alpha_n := \frac{v_n}{\sigma_n^2}.
\]
Put 
\[
\epsilon_n := \alpha_n^{1/3}\sigma_n^2, \ \ \delta_n := \alpha_n^{1/9} m_n, \ \ l_n := [\alpha_n^{-1/24}]. 
\]
Then for sufficiently large $n$, a simple computation gives
\begin{align*}
\sqrt{\frac{v_nm_nl_n^3\log l_n}{\delta_n \epsilon_n }}  + \biggl(\frac{v_n\sigma_n^4l_n^5 \log l_n}{\delta_n^3 m_n \epsilon_n}\biggr)^{1/4} &\le \sqrt{\frac{v_nm_nl_n^4}{\delta_n \epsilon_n }}  + \biggl(\frac{v_n\sigma_n^4l_n^6}{\delta_n^3 m_n \epsilon_n}\biggr)^{1/4}\\
&\le \alpha_n^{7/36} + \alpha_n^{1/48} \frac{\sigma_n}{m_n}.
\end{align*}
Clearly, $\epsilon_n = o(\sigma_n)$, $\delta_n = o(m_n)$, and $l_n \ra \infty$. So the above bound implies multiple peaks as soon as we show that $\sigma_n/m_n$ remains bounded. More is true, though. Let us now prove that $\sigma_n = o(m_n)$. Let $x_+$ denote the positive part of a real number $x$, and let
\[
\psi(x) := \ee(Z-x)_+^2,
\]
where $Z\sim N(0,1)$. Clearly, $\psi$ is everywhere positive and continuous, and $\lim_{x\ra \infty} \psi(x) = 0$. For each $n$, suppose $i_n$ is a coordinate such that $\var(X_{n,i_n})=\sigma_n^2$. Since surely $M_n \ge X_{n,i_n}$, we have 
\begin{align*}
v_n &= \ee(M_n - m_n)^2 \ge \ee(M_n - m_n)_+^2\\
&\ge \ee(X_{n,i_n} - m_n)_+^2 = \sigma_n^2 \psi(m_n/\sigma_n).
\end{align*}
Since $v_n/\sigma_n^2 \ra 0$, this shows that $m_n/\sigma_n \ra \infty$, and completes our argument for chaos $\Rightarrow$ MP using Theorem \ref{p2p3}.

Let us now give a sketch of the proof of Theorem \ref{p2p3}. The idea is very simple: due to the chaos property, we can perturb the field a little bit to get a new maximum at a location `far away' from the original maximum. However, since the perturbation is small and the value of the maximum is concentrated, it follows that the new location of the maximum must have been a near-maximal location in the unperturbed field. This shows the existence of at least two near-maximal points that are `far away' from each other. Repeating the process as many times as we can, we get a large number of near-maxima. 
\begin{proof}[Proof of Theorem \ref{p2p3}]
Let 
\[
t:= \sqrt{\frac{v\delta \log l}{lm\epsilon}}.
\]
Since $\delta < m$, we see that the quantity \eqref{prob} is negative if $ltm/\delta> lt > 1$. Thus, we can assume without loss of generality that $lt \le 1$. Then for any $1\le k\le l$, we have the crude bound
$1-e^{-kt} \ge kt/2$.
Thus,  by Theorem \ref{varconv}, 
\begin{equation}\label{eq1}
\sum_{k=1}^l \ee(R(I^0,I^{kt})) \le \sum_{k=1}^l \frac{v}{1-e^{-kt}}\le \sum_{k=1}^l\frac{2v }{kt} \le \frac{2v\log l}{t}.
\end{equation}
Let $\bbx^{(1)},\ldots,\bbx^{(l)}$ be i.i.d.\ copies of $\bbx$, and recursively define $\bbz^{(0)},\ldots,\bbz^{(l)}$ as follows. Let $\bbz^{(0)} := \bbx$, and for each $k$, let
\[
\bbz^{(k)} := e^{-t} \bbz^{(k-1)} + \esst \bbx^{(k)}.
\]
For each $k$, let
\[
L_k := \argmax_{i\in S} Z^{(k)}_i.
\]
It is easy to see by induction that $\bbz^{(k)}$ has the same distribution as $\bbx$ for every $k$, and 
\[
\cov(\bbz^{(j)}, \bbz^{(k)}) = e^{-t|j-k|} R.
\]
This shows that 
\[
(\bbz^{(j)}, \bbz^{(k)}) \stackrel{dist.}{=} (\bbx^0, \bbx^{|j-k|t}). 
\]
In particular, 
\[
\ee(R(L_j,L_k)) = \ee(R(I^0, I^{|j-k|t})).
\]
Thus, using \eqref{eq1} and the positivity of $R(i,j)$ we get that 
\begin{equation}\label{eq3}
\begin{split}
\pp\bigl(\max_{1\le j\ne k\le l} R(L_j,L_k) \ge \epsilon \bigr) &\le \frac{1}{\epsilon} \sum_{1\le j\ne k\le l} \ee(R(L_j,L_k))\\
&\le \frac{l}{\epsilon}\sum_{k=1}^l \ee(R(I^0, I^{kt})) \le \frac{2vl\log l}{\epsilon t}.
\end{split}
\end{equation}
Now fix some $k\le l$. Note that $\bbz^{(k)}$ can be written as
\[
\bbz^{(k)} = e^{-kt} \bbx + \sqrt{1-e^{-2kt}} \bby,
\]
where $\bby$ is an independent copy of $\bbx$. Thus, if we let 
\[
\bbw := \sqrt{1-e^{-2kt}} \bbx - e^{-kt} \bby,
\]
then $\bbz^{(k)}$ and $\bbw$ are independent (because $\cov(\bbz^{(k)}, \bbw) = \mathbf{0}$), and
\[
\bbx = e^{-kt} \bbz^{(k)} + \sqrt{1-e^{-2kt}} \bbw.
\]
Therefore,
\begin{align*}
\ee|X_{L_k} - m| &\le \ee|e^{-kt}Z^{(k)}_{L_k} - m| + \sqrt{1-e^{-2kt}}\ee|W_{L_k}|\\
&\le e^{-kt} \ee|Z^{(k)}_{L_k} - m| + (1-e^{-kt})m +  \sqrt{1-e^{-2kt}}\ee|W_{L_k}|\\
&\le \sqrt{v} + ktm + \sqrt{2kt}\sigma .
\end{align*}
Note that we crucially used the independence of $\bbz^{(k)}$ and $\bbw$ in the last line to conclude that $\ee|W_{L_k}|\le\sigma$. From the above bound and Markov's inequality, we get
\begin{align*}
\pp\bigl(\max_{1\le k\le l} |X_{L_k} - m| \ge {\textstyle\frac{1}{2}}\delta \bigr) &\le 
\sum_{1\le k\le l} \pp\bigl(|X_{L_k}-m|\ge {\textstyle\frac{1}{2}}\delta )\\
&\le \frac{2(l\sqrt{v} + l^2tm + l^{3/2}\sigma \sqrt{2t})}{\delta}.
\end{align*}
Again, applying Markov's inequality we have
\begin{equation*}\label{eq4}
\pp(|M- m| \ge {\textstyle\frac{1}{2}}\delta) \le \frac{2\sqrt{v}}{\delta}. 
\end{equation*}
Combining the last two inequalities and the inequality~\eqref{eq3}, we get
\begin{align*}
&\pp\bigl(\max_{1\le j\ne k\le l} R(L_j,L_k) \ge \epsilon \ \ \text{or} \  \max_{1\le k\le l} |X_{L_k} - m| \ge {\textstyle\frac{1}{2}}\delta \ \ \text{or} \ \ |M- m| \ge {\textstyle\frac{1}{2}}\delta)\\
&\le  \frac{2vl\log l}{\epsilon t} + \frac{2(l\sqrt{v} + l^2tm + l^{3/2}\sigma \sqrt{2t})}{\delta} + \frac{2\sqrt{v}}{\delta}\\
&= \sqrt{\frac{16vml^3\log l}{\delta \epsilon }}  + \frac{(2l+2)\sqrt{v}}{\delta} + \biggl(\frac{4v\sigma^4l^5 \log l}{\delta^3 m \epsilon}\biggr)^{1/4}.
\end{align*}
Now, since $\epsilon < \sigma^2$ and $l\ge 2$, the ratio of the second term to the third is bounded by
\begin{align*}
\biggl(\frac{64vm \epsilon}{\delta \sigma^4l\log l}\biggr)^{1/4} \le \biggl(\frac{64vm }{\delta \epsilon l\log l}\biggr)^{1/4} \le \biggl(\frac{16vm l^3\log l}{\delta \epsilon}\biggr)^{1/4},
\end{align*}
which is precisely the square-root of the first term. Without loss of generality, we can assume that the first term is $\le 1$ (because we are bounding a probability), and therefore the third term dominates the second. 

The proof is finished by defining $A:= \{L_1,\ldots,L_l\}$. Note that although the construction of the set $A$ involved auxiliary randomness, the {\it existence} of a set like $A$ depends just on the vector $\bbx$, and hence the probability of existence is not affected when we expand the probability space.
\end{proof}

\subsection{Multiple Peaks does not imply Chaos}
In this subsection, we exhibit a sequence of Gaussian fields $(\bbx_n)_{n\ge 1}$ that satisfies the Multiple Peaks condition but is not chaotic. We will also arrange it so that the field has positive correlations between sites and the same variance at each site, just to show that these two factors don't play any role.

For simplicity, we fix $n$ and avoid putting it as a subscript. Let 
\[
(g_{ij}^k, 1\le i\le n, \ j\in \{0,1\}, \ 1\le k\le n)
\]
be a collection of i.i.d.\ standard Gaussian random variables. Let $T$ be the set of all maps from $\{1,\ldots,n\}$ into $\{0,1\}$. For each $f\in T$ and each $k\le n$, define
\[
Y_f^k := \frac{\sum_{i=1}^n g^k_{if(i)}}{\sqrt{n}}.
\]
Let $(Z_k)_{k\le n}$ be i.i.d.\ standard Gaussian random variables, and let 
\[
\rho := 1- n^{-1/3}.
\]
Finally, define the vector $\bbx=(X_f^k)_{f\in T, \; k\le n}$ as
\[
X_f^k = 
\begin{cases}
Y_f^k &\text{ if } k = 1,\\
\rho Y_f^k + \sqrt{1 -\rho^2} Z_k & \text{ if } k > 1.
\end{cases}
\]
Then note that for each $f, f'\in T$ and $k, k'\le n$,
\[
\var(X_f^k) = 1, \ \ R((k,f), (k',f')) := \cov(X_f^k, X_{f'}^{k'}) \ge 0. 
\]
Thus, in our notation, we have
\[
\sigma_n^2 = \max_{f\in T, \ k\le n} \var(X_f^k) = 1.
\]
We claim that $\bbx$ has multiple peaks but is not chaotic. First, we show the existence of multiple peaks. Observe that for any $k\le n$,
\begin{equation}\label{ymax}
\max_f Y_f^k = \frac{\sum_{i=1}^n \max\{g_{i0}, g_{i1}\}}{\sqrt{n}}.
\end{equation}
Thus, if 
\[
\mu:= \ee\max\{g_{10}, g_{11}\},
\]
and
\[
M_k := \max_f X_f^k, 
\]
then we have
\[
\ee\bigl(M_k) = 
\begin{cases}
\mu\sqrt{n} &\text{ if } k =1,\\
\rho\mu \sqrt{n} &\text{ if } k>1.
\end{cases} 
\]
Note that
\[
\max_k M_k =  \max_{f\in T, \; 1\le k\le n} X_f^k =:M.
\]
In particular,
\begin{equation*}\label{mbd}
m_n:=\ee(M) \ge \mu \sqrt{n}.
\end{equation*}
By Proposition \ref{borell}, for each $k$ we have
\[
\pp(|M_k - \ee(M_k)| \ge 2\mu n^{1/6}) \le 2e^{-2\mu^2n^{1/3}}.
\]
Thus,
\[
\pp(\max_k |M_k - \ee(M_k)| \ge 2\mu n^{1/6}) \le 2ne^{-2\mu^2n^{1/3}}.
\]
Now, 
\begin{align*}
\max_k |M_k - \mu \sqrt{n}|&\le \max_k |M_k - \ee(M_k)| + \max_k |\ee(M_k) - \mu\sqrt{n}|\\
&\le \max_k |M_k - \ee(M_k)| + \mu n^{1/6}.
\end{align*}
Therefore,
\begin{equation}\label{p4bd}
\pp\bigl(\max_k |M_k - \mu\sqrt{n}| \ge 3\mu n^{1/6}\bigr) \le 2n e^{-2\mu^2 n^{1/3}}.
\end{equation}
Now choose
\[
A = \{(1,f_1),(2,f_2),\ldots,(n,f_n)\},
\]
where $f_k$ is the $f$ that maximizes $X^k_f$. 
By \eqref{p4bd}, if we take $\gamma_n = 2ne^{-2\mu^2n^{1/3}}$, then with probability $\ge 1-\gamma_n$, we have 
\[
\text{for all $k$,} \ \ X_{f_k}^k = M_k \ge M - 6\mu n^{1/6}.
\]
Also note that $|A|=n$ and $R((i,f_i),(j,f_j)) = 0$ for all $i\ne j$. Thus, the multiple peaks  condition holds with $l_n = n$, $\epsilon_n = 0$, $\delta_n = 6\mu n^{1/6}$, and $\gamma_n$ as above. Clearly, $l_n \ra \infty$, $\epsilon_n = o(\sigma_n^2)$, $\delta_n = o(m_n)$, and $\gamma_n \ra 0$. 

Next, let us show that $\bbx$ is not chaotic. We accomplish this by showing the $\bbx$ is not superconcentrated. Since
\[
\rho\mu \sqrt{n} = \mu \sqrt{n} - \mu n^{1/6},
\]
Proposition \ref{borell} implies that
\begin{align*}
\pp(M_1 < M) &\le \pp\bigl(M_1 < \mu \sqrt{n} - \frac{\mu}{2}n^{1/6}\bigr) \\
&\quad + \pp\bigl(\max_{2\le k\le n} M_k > \rho \mu \sqrt{n} + \frac{\mu}{2}n^{1/6}\bigr)\\
&\le n e^{-\frac{1}{8}\mu^2 n^{1/3}}.
\end{align*}
Therefore,
\begin{align*}
v_n := \var(M) &\ge \ee((M-\ee(M))^2; M_1 = M)\\
&= \ee(M_1 - \ee(M))^2 - \ee((M_1 - \ee(M))^2; M_1 < M)\\
&\ge \var(M_1) - \bigl[\ee(M_1-\ee(M))^4 \pp(M_1 < M)]^{1/2}.
\end{align*}
From \eqref{ymax} we see that $\var(M_1) = \var(\max\{g_{10}, g_{11}\})$. From Proposition~\ref{borell}, it follows that $\ee(M_1-\ee(M_1))^4 \le C$, where $C$ is a universal constant. Again, by an argument similar to Lemma \ref{max} (using the Gaussian tail bounds for $M_1,\ldots,M_n$ that we get by Proposition \ref{borell}), we can show that $|\ee(M_1)-\ee(M)|\le C\sqrt{\log n}$. Combining these observations and the exponentially decaying bound on $\pp(M_1< M)$ obtained above, we see that $v_n$ can be bounded below by a positive constant that does not depend on $n$. Since $\sigma_n^2 \equiv 1$, this completes the argument.

\subsection{Chaos does not imply Strong Multiple Peaks}
In this section, we produce a sequence of Gaussian fields that is superconcentrated, but does not have multiple peaks in the strong sense. As in the previous section, we fix $n$ and avoid putting it as a subscript. 
Let $(g_{ij})_{1\le i,j\le n}$ be a collection of i.i.d.\ standard Gaussian random variables. Let $\mf$ be the set of all functions from $\{1,\ldots,n\}$ into itself. For each $f\in \mf$, let
\[
X_f := \frac{\sum_{i=1}^n g_{if(i)}}{\sqrt{n}}.
\]
Then $\bbx := (X_f)_{f\in \mf}$ is a centered Gaussian random vector, with $\ee(X_f^2) = 1$ for each $f$, and for any $f,f'$,
\[
R(f,f') :=  \cov(X_f, X_{f'}) = \frac{|\{i: f(i)=f'(i)\}|}{n}\in [0,1].
\]
Now, if we let
\[
M_i := \max_{1\le j\le n} g_{ij},
\]
then clearly,
\[
M:= \max_f X_f = \frac{\sum_{i=1}^n M_i}{\sqrt{n}}.
\]
Since $M_1,\ldots,M_n$ are i.i.d.\ and $\var(M_i)\le C/\log n $ for some universal constant $C$ (well-known result; see Proposition \ref{corrbd}), therefore we have
\[
\var(M)\le \frac{C}{\log n}.
\]
This shows that $\bbx$ is superconcentrated. Let us now show that multiple peaks are not present in the strong sense. We need the following simple lemma about the difference between the largest and second largest values in a realization of $n$ i.i.d.\ standard Gaussian random variables.
\begin{lmm}
Suppose $Z_1,\ldots,Z_n$ are i.i.d.\ standard Gaussian random variables. Let $D$ denote the difference between the largest and second largest of these values. There is a constant $\rho >0$ such that for any $n$,
\[
\pp\biggl(D > \frac{\rho}{\sqrt{\log n}}\biggr) >  \frac{7}{8}.
\]
\end{lmm}
\begin{proof}
Let $M_n$ and $M_n^-$ denote the largest and second largest values among $Z_1,\ldots,Z_n$. Let
\[
b_n = \sqrt{2\log n} - \frac{\log \log n + \log (4\pi)}{2\sqrt{2\log n}}, \ \ a_n = \sqrt{2\log n}.
\]
Let $\Phi(x) = \pp(Z_1 \le x)$. Then clearly, for any $-\infty < y < x < \infty$,
\[
\pp(M_n^- \le x,\; M_n > y) = n (1-\Phi(y))\Phi(x)^{n-1}.
\]
Using this and the Mills ratio bounds from Section \ref{basicfacts}, it is easy to show that for any $-\infty < u< v< \infty$,
\[
\lim_{n\ra\infty}\pp\bigl(a_n(M_n^- - b_n) \le u, \; a_n(M_n - b_n) > v\bigr) = e^{-u} e^{-e^{-v}}.
\]
From this, it can deduced that $a_n(M_n - M_n^-)$ converges in law to a distribution that has a positive density on $(0,\infty)$ and no point masses. This proves the lemma.
\end{proof}

Let us now finish the proof of the nonexistence of strong multiple peaks in the field $\bbx$. For each $i$, let $M_i^-$ be the second largest value among $g_{i1},\ldots,g_{in}$. Let $D_i := M_i - M_i^-$. Let $D_{(1)} \le D_{(2)}\le \cdots \le D_{(n)}$ denote the $D_i$'s arranged in increasing order. 
Let $\rho$ be the constant from the above lemma. Then by the lemma and the weak law of large numbers we see that as $n \ra \infty$,
\begin{equation*}\label{ordbd0}
\lim_{n \ra \infty} \pp\biggl( \frac{1}{n}\biggl|\biggl\{i: D_i > \frac{\rho}{\sqrt{\log n} }\biggr\}\biggr| \ge \frac{7}{8}\biggr) = 1. 
\end{equation*}
(Note that we should have written $D_{n,i}$ instead of $D_i$, but we will be slack about this kind of thing.) It is not difficult to see from here that
\begin{equation}\label{ordbd}
\lim_{n\ra \infty} \pp\biggl(\sum_{i=1}^{[n/4]}D_{(i)} \ge \frac{\rho n}{8\sqrt{\log n}}\biggr) = 1.
\end{equation}
Let $f^*$ denote the maximizing function, that is $f^*(i) = \argmax_j g_{ij}$. Suppose $f,f'$ are two functions such that $R(f,f') < 1/2$. Then we must have that $\min\{R(f,f^*), R(f', f^*)\} < 3/4$, because if $R(f,f^*)\ge 3/4$ and $R(f',f^*)\ge 3/4$,  then $R(f,f') \ge 1/2$. Thus,
\begin{equation}\label{ordbd2}
X_{f^*} - \min\{X_f, X_{f'}\} \ge \frac{\sum_{i=1}^{[n/4]} D_{(i)}}{\sqrt{n}}.
\end{equation}
Let $E$ denote the following event: For any  $f, f'\in \mf$ such that $R(f,f')< 1/2$, at least one of $X_f$ and $X_{f'}$ is less than
\[
X_{f^*} - \frac{\rho}{8}\sqrt{\frac{n}{\log n}}.
\]
Then from \eqref{ordbd} and \eqref{ordbd2}, we see that $\pp(E)\ra 1$ as $n \ra\infty$. This shows that the Strong Multiple Peaks condition does not hold for the sequence of Gaussian fields under consideration. 

\section{Hypercontractivity}\label{hypercon}
Suppose we have a semigroup of operators $(P_t)_{t\ge 0}$, acting on some space of functions on $\rr^n$ (semigroup means $P_tP_s = P_{t+s}$). Suppose $\mu$ is an invariant measure for the semigroup, meaning that $\int P_t f d\mu = \int f d\mu$ for any $f$ in the domain of $P_t$. The semigroup is said to be `hypercontractive' if for any $t >0$, there are numbers $q > p > 1$ (possibly depending on $t$) such that $P_t$ maps $L^{q}(\mu)$ into $L^p(\mu)$ and moreover, $\|P_tf \|_{q} \le \|f\|_p$ for all $f\in L^q(\mu)$. Here $\|\cdot\|_p$ denotes the standard $L^p$-norm on $L^p(\mu)$. This phenomenon was discovered by Nelson \cite{nelson66} and established for the Ornstein-Uhlenbeck semigroup by him a few years later  \cite{nelson73}. The O-U semigroup on $\rr^n$ is defined as follows. Let $\bbz$ be a standard Gaussian random vector in $\rr^n$. For any function $f$ and $t\ge 0$, define the function $P_tf$ as 
\[
P_tf(\bx) := \ee f(\est \bx + \esst \bbz). 
\]
It is easy to see that the standard Gaussian measure on $\rr^n$, which we denote by $\gamma^{\otimes n}$,  is an invariant measure for this semigroup. 
Nelson \cite{nelson73} showed that for any $p > 1$ and $t\ge 0$, if we let $q = 1+e^{2t} (p-1) > p$, then for all $f\in L^q(\gamma^{\otimes n})$, we have
\begin{equation}\label{hyperbd}
\|P_t f\|_q \le \|f\|_p.
\end{equation}
Note that the result does not depend on the dimension $n$ at all. This is one of the remarkable properties that make hypercontractivity a deep and powerful tool. 

The study of hypercontractive semigroups was given a major boost by the discovery of the connection between hypercontractivity and logarithmic Sobolev inequalities by Gross \cite{gross75}. For surveys of the extensive literature that developed around this topic, one can look in the wonderful monographs~\cite{aneetal00} and \cite{guionnetzegarlinski03}. 

The connection between logarithmic Sobolev inequalities (and hence hypercontractivity) and concentration of measure was discovered by I.~Herbst in a small unpublished note (see Ledoux \cite{ledoux01}, Theorem 5.3). However, the Herbst argument can only establish ordinary concentration of measure. The first application of hypercontractivity to prove what we call superconcentration was due to Talagrand~\cite{talagrand94}. Talagrand's method was subsequently used by Benjamini, Kalai, and Schramm \cite{bks03} to prove superconcentration in first passage percolation (they call it `sublinear variance'). 

The way we use hypercontractivity in this paper to establish superconcentration is essentially the same in spirit as Talagrand's technique, although there is the difference that while Talagrand's method applies only to functions of independent random variables, we can deal with strong dependence as in the Gaussian free field. 

Recall that in our setting, we have a vector $\bbx = (X_i)_{i\in S}$ that is centered Gaussian but not necessarily with independent components. We can still define a semigroup $(\tilde{P}_t)_{t\ge 0}$ as
\[
\tilde{P}_t f(\bx) := \ee f(\est \bx + \esst \bbx).
\]
Note that this semigroup retains the same hypercontractive property \eqref{hyperbd} as the standard O-U semigroup, with exactly the same relation between $p$ and $q$. This can be argued as follows. In our notation, $R = (R(i,j))_{i,j\in S}$ is the covariance matrix of $\bbx$. Assuming that the vector $\bbx$ is not identically equal to zero, we know that for some $d\le |S|$, there is a $|S|\times d$ matrix $B$ of full rank such that $BB^T = R$. We can assume that there is a $d$-dimensional  standard Gaussian vector $\bbz$ such that $\bbx = B\bbz$. Let $\bbz'$ be an independent copy of $\bbz$, and let $\bbx' = B\bbz'$. Given $f:\rr^S \ra \rr$, define the function $g:\rr^d\ra \rr$ as 
\[
g(\bx) := f(B\bx).
\]
The random vector $\bbx$ is supported on the image of $\rr^d$ under the map $B$. On this subspace the map $B$ has an inverse, which we denote by $B^{-1}$. Then we have 
\begin{align*}
\tilde{P}_t f(\bx) &= \ee f(\est \bx + \esst \bbx) \\
&= \ee f(\est BB^{-1} \bx + \esst B \bbz)\\
&= \ee g(\est B^{-1} \bx + \esst \bbz) = P_t g(B^{-1} \bx). 
\end{align*}
Therefore, we have that for every $t\ge 0$,
\[
\tilde{P}_t f(\bbx) = P_t g(B^{-1}\bbx) = P_tg(\bbz),
\] 
and so by \eqref{hyperbd}, 
\[
\|\tilde{P}_t f(\bbx)\|_q = \|P_t g(\bbz)\|_q \le \|g(\bbz)\|_p = \|g(\bbx)\|_p.
\]
Thus, the hypercontractive property \eqref{hyperbd} holds for the semigroup $(\tilde{P}_t)_{t\ge 0}$ with the same $q,p$. The following lemma, which is of crucial importance to us, is a direct consequence. Recall all notation from Section \ref{intro}, including the definition of $\bbx^t$.
\begin{lmm}\label{hyper}
For any measurable $f:\rr^{S}\ra [0,1]$ and any $t\ge 0$, we have
\[
\ee(f(\bbx)f(\bbx^t))\le (\ee f(\bbx))^{1+\tanh(t/2)}.
\]
\end{lmm}
\begin{proof}
Take any $p> 1$ and let $q = 1+ e^{2t}(p-1)$. Let $q' = q/(q-1)$. Since $f$ maps into $[0,1]$, we have
\begin{align*}
\ee(f(\bbx) f(\bbx^t)) =  \ee(f(\bbx) \tilde{P}_t f(\bbx))  &\le \|f(\bbx)\|_{q'} \|\tilde{P}_tf(\bbx)\|_{q}\\
&\le \|f(\bbx)\|_{q'} \|f(\bbx)\|_{p}\\
&\le \bigl(\ee f(\bbx)\bigr)^{\frac{1}{q'} + \frac{1}{p}}.
\end{align*}
We now optimize over $p$, which yields
\[
p = 1 + e^{-t}.
\]
An easy computation gives
\begin{align*}
\frac{1}{q'} + \frac{1}{p} &= 1 - \frac{1}{1+e^{2t}(p-1)} + \frac{1}{p}\\
&= 1 + \tanh(t/2).
\end{align*}
This completes the proof.
\end{proof}

Lemma \ref{hyper} is a very useful tool, as we will see in our proof of superconcentration in directed polymers in Section~\ref{polymer} and in the Kauffman-Levin $NK$ model in Section \ref{nk}. It is particularly potent in combination with Lemma~\ref{intparts} or Theorem \ref{varmax}. To demonstrate an immediate application, let us work out the following simple result that we do not know how to prove without using hypercontractivity. It says that the variance of the maximum is small whenever the correlations are uniformly small. Recall that $v = \var(\max_i X_i)$ and $R(i,j) = \cov(X_i,X_j)$. 
\begin{prop}\label{corrbd}
Suppose that $R(i,j)\le \rho$ for each $i\ne j$ and $R(i,i) = \sigma^2$ for each $i$. Then 
\[
v \le \frac{C\sigma^2}{\log|S|} + C\rho,
\]
where $C$ is a universal constant. 
\end{prop}
\noindent{\it Remark.} It can be easily shown that the bound is sharp by considering the case where $R(i,j) = \rho$ for each $i\ne j$. However, in spite of the sharpness, the result is probably not very useful since the uniform boundedness is a very strong restriction. It is presented here only for illustration purposes.
\begin{proof}
For any vector $\bx \in \rr^S$ all of whose coordinates have distinct values, define  
\[
f_i(\bx) := \ii_{\{x_i = \max_j x_j\}}.
\]
Then by Lemma \ref{hyper} and the hypothesis of this theorem, we have
\begin{align*}
\ee(R(I^0,I^t)) &\le \sigma^2\pp(I^0 = I^t) + \rho \pp(I^0\ne I^t)\\
&= \sigma^2\sum_{i\in S}  \ee(f_i(\bbx^0) f_i(\bbx^t)) + \rho\pp(I^0\ne I^t)\\
&\le \sigma^2 \sum_{i\in S} (\ee f_i(\bbx^0))^{1+\tanh(t/2)} + \rho\\
&\le \sigma^2 (\max_{i\in S} \pp(I^0 = i))^{\tanh(t/2)} \sum_{i\in S} \pp(I^0 = i) + \rho \\
&= \sigma^2 (\max_{i\in S} \pp(I^0 = i))^{\tanh(t/2)}  + \rho. 
\end{align*}
Now, by increasing the value of the constant in the statement of the theorem (and recalling Proposition \ref{varcrude}), we can assume without loss of generality that $\rho \le \sigma^2/2$. Under this assumption, by the Sudakov minoration technique stated in Section \ref{basicfacts}, we have
\[
m \ge C\sigma \sqrt{\log |S|},
\]
where we are using our convention of letting $C$ denote any universal constant. Thus, by Proposition \ref{borell} and the tail bound \eqref{gausstail} we have
\begin{align*}
\pp(I^0 = i) &\le \pp(X_i \ge m/2) + \pp(M \le m/2)\\
&\le 2e^{-m^2/8\sigma^2}\le |S|^{-C}. 
\end{align*}
Now, for each $t\ge 0$,
\begin{equation}\label{tanhbd}
\tanh(t/2) = \frac{e^{t/2} - e^{-t/2}}{e^{t/2} + e^{-t/2}} = \frac{(1-e^{-t/2})(1+e^{-t/2})}{1+e^{-t}} \ge 1-e^{-t/2}.
\end{equation}
Combining this with the bounds on $\ee(R(I^0,I^t))$ and $\pp(I^0 =i)$ obtained above and Theorem \ref{varmax}, we get
\begin{align*}
v &= \int_0^\infty e^{-t} \ee(R(I^0,I^t)) dt\\
&\le \int_0^\infty e^{-t} \bigl(\sigma^2|S|^{-C(1-e^{-t/2})} + \rho\bigr) dt\\
&\le \int_0^\infty e^{-t/2} \sigma^2|S|^{-C(1-e^{-t/2})} dt + \rho\\
&= 2\sigma^2 \int_0^1 |S|^{-Cu} du + \rho \le \frac{C\sigma^2}{\log |S|} + \rho. 
\end{align*}
This completes the proof. 
\end{proof}
A part of the proof of Proposition \ref{corrbd} generalizes to the following technique, that will be help us bound the variance of the maximum of a discrete Gaussian free field in Section \ref{dgff}. It will also be used for analyzing continuous Gaussian fields on Euclidean spaces in Section \ref{euclidean}.
\begin{thm}\label{hyper2}
Suppose that for each $r\ge 0$, there is a covering $\mc(r)$ of $S$ such that whenever $i,j$ are indices with $R(i,j)\ge r$, we have $i,j\in D$ for some $D\in \mc(r)$.
For each $A\subseteq S$ let $p(A) := \pp(I^0\in A)$, and define 
\[
\rho(r) := \max_{D\in \mc(r)} p(D), \ \ \ \mu(r) := \sum_{D\in \mc(r)} p(D) = \ee\bigl|\{D\in \mc(r): I^0\in D\}\bigr|.
\] 
Then we have
\[
v\le \int_0^{\sigma^2} \frac{2\mu(r)(1-\rho(r))}{-\log \rho(r)}\; dr,
\]
where we interpret $(1-x)/\log x = -1$ when $x=1$. 
\end{thm}
\begin{proof}
By Lemma \ref{hyper}, we have
\begin{align*}
\pp(R(I^0, I^t) \ge r) &\le \sum_{D\in \mc(r)} \pp(I^0\in D, \ I^t\in D)\\
&\le \sum_{D\in\mc(r)}p(D)^{1+\tanh(t/2)}\\
&\le \rho(r)^{\tanh(t/2)}\sum_{D\in \mc(r)} p(D) = \mu(r)\rho(r)^{\tanh(t/2)}.
\end{align*}
Thus, by Lemma \ref{varmax} we have
\begin{align*}
v &= \int_0^\infty e^{-t} \ee(R(I^0,I^t))\; dt\\
&\le \int_0^\infty \int_0^{\sigma^2} e^{-t} \pp(R(I^0,I^t)\ge r) \; dr\; dt\\
&\le \int_0^{\sigma^2} \int_0^\infty e^{-t} \mu(r) \rho(r)^{\tanh(t/2)}\; dt\; dr.
\end{align*} 
Now, by \eqref{tanhbd} we have that for any fixed $r$,
\begin{align*}
\int_0^\infty e^{-t} \rho(r)^{\tanh(t/2)}\; dt &\le \int_0^\infty e^{-t} \rho(r)^{1-e^{-t/2}}\; dt\\
&\le \int_0^\infty e^{-t/2} \rho(r)^{1-e^{-t/2}}\; dt\\
&= \int_0^1 2\rho(r)^u\; du = \frac{2(1-\rho(r))}{-\log \rho(r)}.
\end{align*}
This completes the proof.
\end{proof}

\section{Extremal fields}\label{proof:main2}
The goal of this section is to prove Theorems \ref{main2} and \ref{gencorr}. As usual, we prove quantitative versions. Recall the definitions of $\bbx, \bbx^t, M, m, v, \sigma^2, I^t,$ and $R(i,j)$ from Section \ref{intro}. The following theorem is the main result of this section.
\begin{thm}\label{extreme}
Let
\[
\alpha := \frac{m}{\sqrt{2\sigma^2\log |S|}}, \ \ \beta := \sqrt{1-\alpha} + \biggl(\frac{\log \log |S|}{ \log |S|}\biggr)^{1/4}.
\] 
Then $v\le C\sigma^2 \beta$ and for any $t\ge 0$, $\ee(R(I^0,I^t)) \le \frac{C\sigma^2\beta}{\esst}$, where $C$ is a universal constant.
\end{thm}
\noindent It is clear how Theorem~\ref{main2} follows from this result. However, Theorem~\ref{gencorr} still needs a proof. Let us prove it before moving on to prove Theorem \ref{extreme}. As usual $C$ will denote any universal constant, whose value may change from line to line.
\begin{proof}[Proof of Theorem~\ref{gencorr}]
Without loss of generality, assume that $\sigma^2 = 1$. Let $n = |S|$, and define
\[
N := |\{i: X_i \ge  \sqrt{2\log n}\}|.
\]
Then by the Mills ratio lower bound from Section \ref{basicfacts} and the assumption that $R(i,i)=1$ for all $i$,  we have
\begin{align*}
\ee(N) &\ge Cn \frac{e^{-\log n}}{\sqrt{2\log n}} = \frac{C}{\sqrt{2\log n}}.
\end{align*}
Again, we have
\begin{align*}
\ee(N^2) &= \sum_{i,j\in S}\pp\bigl(X_i \ge \sqrt{2\log n}, \ X_j \ge \sqrt{2\log n}\bigr) \\
&\le \sum_{i,j\in S}\pp\bigl(X_i + X_j \ge 2\sqrt{2\log n}\bigr)\\
&\le  \sum_{i,j\in S}\exp\biggl(-\frac{4\log n}{\var(X_i + X_j)}\biggr)\\
&\le  \sum_{i,j\in S}\exp\biggl(-\frac{2\log n}{1+R(i,j)}\biggr).
\end{align*}
Applying the Paley-Zygmund second moment method, we get
\[
\pp\bigl(M \ge \sqrt{2\log n}\bigr) = \pp(N>0) \ge \frac{(\ee(N))^2}{\ee(N^2)} \ge \frac{C}{\sum_{i,j\in S} n^{-2/(1+R(i,j))} \log n}.
\]
Again, by Proposition \ref{borell} we know that for any $x\ge 0$,
\[
\pp(M - \ee(M)\ge x) \le e^{-x^2/2}.
\]
Combining, we see that
\[
\ee(M) \ge \sqrt{2\log n} - C\biggl(\log \log n + \log \sum_{i,j\in S}n^{-2/(1+R(i,j))}\biggr)^{1/2}.
\]
The proof now follows from Theorem \ref{extreme}.
\end{proof}

Let us now give a brief idea of the proof of Theorem \ref{extreme} before going into the technical details. The purpose of this sketch is to succinctly convey an idea that may have the potential to grow as an alternative to the hypercontractive method for proving superconcentration. We begin by assuming that $\sigma^2 = 1$. The first step is to show that for any $t\ge 0$, $X_{I^t} \approx e^{-t} m$ with high probability. We carry out this important step in Lemma~\ref{prediction}. Next, we fix $t$ and let $D = \{i: X_i \approx \est m\}$, where the meaning of $\approx$ has to be made precise. A simple first moment bound shows that with high probability, $|D|\lesssim n^{1 - e^{-2t}}$. Next we fix some $r\ge 0$ and let $B=\{i: R(I_0,i)\ge r\}$. The key observation is that since $\bbx'$ and $\bbx$ are independent and $\sigma^2 = 1$, we have $\var(X'_i - rX'_{I^0}\mid \bbx) \le 1-r^2$ for each $i\in B$, and hence 
\[
\ee\bigl(\max_{i\in B\cap D} X_i'\mid \bbx\bigr) = \ee\bigl(\max_{i\in B\cap D} (X_i' - rX'_{I^0})\mid \bbx\bigr) \lesssim \sqrt{(1-r^2)2\log |D|}.
\]
Combining this with the bound $|D|\lesssim n^{1-e^{-2t}}$, we get
\begin{align*}
\max_{i\in B\cap D} X^t_i &\le \est \max_{i\in B\cap D} X_i + \esst \max_{i\in B\cap D} X_i'\\
&\lesssim e^{-2t} \alpha \sqrt{2\log n} + (1-e^{-2t}) \sqrt{2(1-r^2)\log n}.
\end{align*}
Thus, if $\sqrt{1-r^2} < \alpha$, or in other words $r > \sqrt{1-\alpha^2}$, we cannot have that $\max_{i\in B\cap D} X_i^t \approx \alpha \sqrt{2\log n}$. If this does not happen, then $I^t\not \in B\cap D$. But we have already stated that $I^t \in D$ with high probability. Therefore, whenever $r > \sqrt{1-\alpha^2}$, we have that $I^t \in D\backslash B$ with high probability, which implies that $R(I^0,I^t) < r$ with high probability. Roughly, this justifies the $\sqrt{1-\alpha}$ term in the statement of the theorem. The second term arises from our attempts at making the above sketch rigorous, which is a somewhat technically involved  task. 

Let us now begin the formal proof. The first step, as mentioned before, is to show that $X_{I^t} \approx e^{-t} m$. This is made precise in the following lemma.
\begin{lmm}\label{prediction}
Take any $t \ge 0$. Then for any $x \ge 0$, we have
\[
\pp\bigl(|X_{I^t} - \est m|\ge x\bigr) \le 4e^{-x^2/4\sigma^2}.
\]
\end{lmm}
\begin{proof}
For notational convenience, let $a=\est$, $b = \esst$,  $\bbz := \bbx^t$, and 
$\bbw := b\bbx - a\bbx'$.
Then $\bbz$ and $\bbw$ are independent, and
\begin{equation}\label{equ0}
\bbx = a\bbz + b\bbw.
\end{equation}
Since $a+b\le \sqrt{2}$, by Proposition \ref{borell} and the independence of $\bbz$ and $\bbw$, we have
\begin{align*}
\pp(|X_{I^t} - am| \ge x) &\le \pp(a|Z_{I^t} - m| \ge ax/\sqrt{2}) + \pp(b|W_{I^t}| \ge bx/\sqrt{2})\\
&\le 2e^{-x^2/4\sigma^2} + 2e^{-x^2/4\sigma^2}.
\end{align*}
This completes the proof.
\end{proof}

\begin{proof}[Proof of Theorem \ref{extreme}]
Without loss of generality, we assume that $\sigma^2 = 1$. 
As in Lemma~\ref{prediction}, let $a= \est$, $b = \esst$, and $\bbz = \bbx^t$. Let $n = |S|$. 
Note that since $|R(i,j)|\le 1$ for all $i,j$ and $C$ can be chosen as large as we like, it suffices to prove the theorem assuming that $n$ is larger than some fixed threshold and that
\begin{equation}\label{b}
\frac{1}{b}\biggl(\frac{\log \log n}{\alpha^2\log n}\biggr)^{1/4} \le \frac{1}{100}.
\end{equation}
For the same reason, there is no loss of generality in assuming that $m\ge 2$. With that assumption, define
\begin{equation}\label{ep}
\delta:= \frac{\sqrt{\log (m^4/2)}}{m} \in (0,1),
\end{equation}
and let
\[
D := \{i: |X_i - am|\le \delta m\}.
\]
By Lemma \ref{prediction},
\begin{equation}\label{equu1}
\pp(I^t \not \in D) \le 4e^{-\delta^2 m^2 /4}.
\end{equation}
Now, if $a\ge \delta$, then by the Mills ratio upper bound from Section \ref{basicfacts}, 
\begin{align*}
\ee|D| &\le \sum_i\pp(X_i \ge (a-\delta)m)\\
&\le n e^{-(a-\delta)^2m^2/2} = n^{1-(a-\delta)^2\alpha^2}.
\end{align*}
Therefore, if we define
\[
\zeta := 
\begin{cases}
1- (a-\delta)^2 \alpha^2 + \delta^2 & \text{ if } a \ge \delta,\\
1 &\text{ if } a<\delta,
\end{cases}
\]
then in all cases we have
\begin{equation}\label{markov}
\pp(|D| > n^\zeta)\le n^{-\delta^2}.
\end{equation}
Define
\[
\gamma := \frac{\alpha(b^2 - 2\delta)}{b\sqrt{\zeta}}.
\]
It is easy to verify using \eqref{b} that $b^2 \ge 2\delta$ and hence $\gamma\ge 0$. Again,
\[
\frac{\alpha(b^2 - 2\delta)}{b\sqrt{\zeta}}\le \frac{b^2 - 2\delta}{b\sqrt{1-a^2}}  = \frac{b^2 - 2\delta}{b^2}\le 1.
\]
Thus, $\gamma \in [0,1]$. Let $r := \sqrt{1-\gamma^2}$, and define the random set
\[
B := \{i: R(I^0,i)\ge r\}.
\]
Note that
\begin{equation}\label{equu2}
\pp(R(I^0,I^t)\ge r) = \pp(I^t\in B) \le \pp(I^t\in B\cap D) + \pp(I^t \not \in D).
\end{equation}
Let $\ee^0$ and $\var^0$ denote the conditional expectation and variance given $\bbx$. Since $R(i,i)\le 1$ and $R(I^0,i) \ge r$ for all $i\in B$ and $\bbx'$ is independent of $\bbx$, we have
\[
\var^0(X'_i - rX'_{I^0}) \le 1+r^2 - 2r^2 = \gamma^2.
\]
Again, $\ee^0(X'_{I^0}) = 0$. Thus, if $B\cap D \ne \emptyset$, then by Lemma \ref{max} we have
\begin{align*}
\ee^0(\max_{i\in B\cap D} X'_i) &= \ee^0(\max_{i\in B\cap D}(X'_i - rX'_{I^0})) \\
&\le \gamma\sqrt{2\log |B\cap D|}\le \gamma\sqrt{2\log |D|}.
\end{align*}
Combined with Proposition \ref{borell}, this implies that if $B\cap D\ne \emptyset$, then for all $x\ge 0$,
\begin{equation}\label{ppp}
\pp^0(\max_{i\in B\cap D} X'_i \ge \gamma \sqrt{2\log |D|} + x) \le e^{-x^2/2}.
\end{equation}
Clearly, the inequality holds even if we relax the condition $B\cap D \ne \emptyset$ to just $D\ne \emptyset$, interpreting the maximum of an empty set as $-\infty$. Since $X_i \le (a+\delta)m$ for $i\in D$ and $\bbz = a\bbx + b \bbx'$, we have
\begin{align*}
\max_{i\in B\cap D} Z_i &\le a\max_{i\in B\cap D} X_i + b\max_{i\in B\cap D} X'_i\\
&\le a(a+\delta) m + b\max_{i\in B\cap D} X'_i.
\end{align*}
Thus, from \eqref{ppp}, we see that whenever $ D \ne \emptyset$, 
\begin{equation}\label{pppp}
\pp^0(\max_{i\in B\cap D} Z_i \ge a(a+\delta) m + b\gamma\sqrt{ 2\log |D|} + bx) \le e^{-x^2/2}.
\end{equation}
Putting
\begin{align*}
m' &:= a(a+\delta) m + b\gamma\sqrt{ 2\zeta \log n}\\
&= \biggl(a(a+\delta)  + \frac{b\gamma\sqrt{\zeta}}{\alpha} \biggr)m\\
&= (a^2 + a\delta + b^2 - 2\delta) m = (1+a\delta -2\delta) m,
\end{align*}
and using \eqref{pppp} and \eqref{markov} we get
\begin{equation}\label{bc}
\begin{split}
\pp(\max_{i\in B \cap D} Z_i \ge m' + bx) &\le \ee(\pp^0(\max_{i\in B \cap D} Z_i \ge m' + bx); 1\le |D|\le n^\zeta) \\
&\quad + \ee(\pp^0(\max_{i\in B \cap D} Z_i \ge m' + bx); |D|> n^\zeta) \\
&\le e^{-x^2/2} + n^{-\delta^2}.
\end{split}
\end{equation}
Now
\begin{equation}\label{mm}
\begin{split}
m - m' &= (2-a)\delta m \ge \delta m.
\end{split}
\end{equation}
In particular, $m'\le m$. Let $x := (m-m')/2b$. Then by \eqref{bc} and Proposition \ref{borell} we have
\begin{equation}\label{imp}
\begin{split}
\pp(I^t \in B\cap D) &\le \pp(\max_{i\in B\cap D} Z_i \ge m' + bx) + \pp(\max_{i\in S} Z_i \le m- bx)\\
&\le e^{-x^2/2} + n^{-\delta^2} + e^{-b^2x^2/2}\\
&\le 2e^{-(m-m')^2/8} + n^{-\delta^2}.
\end{split}
\end{equation}
Thus, by \eqref{equu1}, \eqref{equu2}, \eqref{mm}, and \eqref{imp}, we have
\begin{equation}\label{peq}
\pp(R(I^0, I^t) \ge r) \le 6 e^{-\delta^2 m^2/8} + n^{-\delta^2} \le 7e^{-\delta^2m^2/8}.
\end{equation}
Now, if $a \ge \delta$, then 
\begin{align*}
r^2 &= \frac{b^2\zeta - \alpha^2(b^4 - 4b^2\delta + 4\delta^2) }{b^2\zeta}\\
&= \frac{b^2(1-a^2\alpha^2+ 2a\delta\alpha^2 -\delta^2 \alpha^2 + \delta^2) - \alpha^2(b^4 - 4b^2\delta + 4\delta^2) }{b^2\zeta}\\
&\le \frac{(1-\alpha^2) b^2 + 2b^2a\delta\alpha^2 + b^2\delta^2 + 4\alpha^2b^2\delta}{b^4}\\
&\le \frac{1-\alpha^2 + C\delta}{b^2},
\end{align*}
where $C$ denotes a universal constant (whose value may change from line to line). Again, if $a< \delta$, 
\begin{align*}
r^2 &= \frac{b^2 - \alpha^2(b^4 - 4b^2\delta + 4\delta^2)}{b^2}\\
&\le 1-\alpha^2 b^2+ 4\alpha^2\delta\\
&= 1-\alpha^2 + \alpha^2 a^2 + 4\alpha^2\delta\le 1-\alpha^2 + 5\delta.
\end{align*}
Therefore in all cases we have
\begin{align*}
\ee(R(I^0, I^t)) &\le r + \pp(R(I^0,I^t)\ge r)\\
&\le \biggl(\frac{1-\alpha^2 + C\delta}{b^2}\biggr)^{1/2} + Ce^{-\delta^2m^2/8},
\end{align*}
From the definition \eqref{ep} of $\delta$ we have
\[
e^{-\delta^2 m^2 /8} = \frac{2^{1/8}}{m^{1/2}}.
\]
Since $\sqrt{x+y}\le \sqrt{x}+\sqrt{y}$ and $\alpha \le 1$, we get
\begin{align*}
\ee(R(I^0, I^t)) &\le \frac{\sqrt{2(1-\alpha)}}{b} + \frac{C(\log m)^{1/4}}{b m^{1/2}}.
\end{align*}
Now, $m^{1/2} = \alpha^{1/2}(2\log n)^{1/4}$. Since $R(i,j)\le 1$ for all $i,j$, we can put a large enough constant in front of the first term to remove the $\alpha^{1/2}$ from the denominator of the second term. Finally, the bound on $v$ comes from Lemma \ref{varmax} and our bound on $\ee(R(I^0,I^t))$.  This completes the proof.
\end{proof}

\section{Example: Chaotic nature of the first eigenvector}\label{eigen}
A random Hermitian matrix $A = (a_{ij})_{1\le i,j\le n}$ is said to belong to the Gaussian Unitary Ensemble (GUE) if (i) $(a_{ij})_{1\le i\le j\le n}$ are independent random variables, (ii) the diagonal entries are standard real Gaussian random variables, and (iii) $(a_{ij})_{1\le i< j\le n}$ are standard complex Gaussian random variables (i.e.\ real and imaginary parts are i.i.d.\ $N(0,1/2)$). 

Eigenvalues of GUE matrices are among the most widely studied objects in random matrix theory. For a general introduction to  the classical random matrix ensembles and results, we refer to the book by Mehta~\cite{mehta91}. The study of the largest eigenvalue was revolutionized through the work of Tracy and Widom \cite{tracywidom94a, tracywidom94b, tracywidom96}. One of the striking implications of their work is that the largest eigenvalue has variance of order $n^{-1/3}$, beating the $O(1)$ bound given by standard isoperimetric and martingale methods. But the Tracy-Widom result is in the sense of weak convergence, and does not provide an actual bound on the variance that we need. A variance bound of order $n^{-1/3}$  was proved by Ledoux \cite{ledoux03} and independently by Aubrun \cite{aubrun05}.

The eigenvectors of GUE matrices, taken as rows (or columns) of a matrix, give rise to a Haar-distributed unitary matrix. In that sense, they are quite well-understood. However, the behavior of the eigenvectors under perturbations of the matrix has not been studied. Such questions arise, for instance, in the study of chaos in the spherical Sherrington-Kirkpatrick model of spin glasses. For the definition of this model and further references, let us refer to the recent paper of Talagrand and Panchenko \cite{panchenkotalagrand07}, where it was proved that the model is chaotic with respect to an external field. The goal of this section is to show that the first eigenvector is unstable under small perturbations of the matrix, and give a quantitative version of this statement. In the spherical SK model with complex spins, this establishes chaos with respect to the disorder at zero temperature.

Let us now formulate the question in terms of Gaussian fields. For each vector $\bu$ in the unit sphere $S^{2n-1}$ of $\cc^n$, define the quadratic form
\[
X_\bu := \bu^* A\bu = \sum_{i,j=1}^n a_{ij} \overline{u}_iu_j,
\]
where, as usual, $\overline{u}_j$ is the complex conjugate of $u_j$ and $\bu^*$ is the adjoint (i.e.\ conjugate transposed) of $\bu$. Since $A$ is Hermitian, $X_\bu$ is a real Gaussian random variable.

Now, if $\bv = z\bu$ for some $z$ in the unit sphere $U(1)$ of $\cc$, then $X_{\bv} \equiv X_{\bu}$. Therefore to retain identifiability, we define $X_{\bu}$ for each $\bu$ not in $S^{2n-1}$, but in the complex projective space $\cc\pp^{n-1} = S^{2n-1}/U(1)$. However, we will continue to write elements of $\cc\pp^{n-1}$ as if they were elements of $\cc^n$, with the quotienting being implicit. With that convention, let  
\[
\lambda_1 := {\textstyle \max_{\bu\in \cc\pp^{n-1}}} X_{\bu}, \ \ \ \bu_1 := \argmax_{\bu\in \cc\pp^{n-1}} X_{\bu}.
\]
Then $\lambda_1$ is the largest eigenvalue of the GUE matrix  $A$ and $\bu_1$ is the corresponding unit eigenvector. Our objective in this section is to show that $\bu_1$ is chaotic under small perturbations of $A$. 

Here we must remark that $\bu_1$ is almost surely well-defined in $\cc\pp^{n-1}$. This follows from the fact that the eigenvalues all have multiplicity $1$ almost surely, which can be deduced, for instance, from the well-known joint density of the eigenvalues of GUE (see e.g.\ Mehta \cite{mehta91}, Chapter 3).

Now let $A'$ be an independent copy of $A$, and as usual define the perturbed matrix $A^t := \est A+ \esst A'$. Let $\bu_1^t$ be the first eigenvector of $A^t$. We want to show that $|\smallavg{\bu_1, \bu_1^t}| := |\sum_{i=1}^n u_{1,i} \overline{u}_{1,i}^t|$ tends to decay rapidly with~$t$. Note that there is no ambiguity because for any $\bu,\bv\in \cc\pp^{n-1}$, $|\smallavg{\bu,\bv}|$ is well-defined (although $\smallavg{\bu,\bv}$ is not).
\begin{thm}
There is a universal constant $C$ such that for any $t\ge 0$, 
\[
\ee|\smallavg{\bu_1,\bu_1^t}|^2 \le \frac{C}{(1-e^{-t}) n^{1/3}},
\]
where $\bu_1^t$ is the first eigenvector of the perturbed matrix $A^t$ defined above.
\end{thm}
\noindent {\it Remarks.} This shows that whenever $t\gg n^{-1/3}$, the vectors $\bu_1$ and $\bu_1^t$ are almost orthogonal with high probability. As mentioned before, this result also proves that the ground state of the (complex) spherical Sherrington-Kirkpatrick model is chaotic under small perturbations of the disorder.
\begin{proof}
An easy computation gives that for any $\bu,\bv \in \cc\pp^{n-1}$,
\[
\cov(X_{\bu}, X_{\mathbf{v}}) = |\smallavg{\bu,\bv}|^2,
\]
where one should note that the right hand side is well-defined on $\cc\pp^{n-1}$.
It is known from random matrix theory \cite{ledoux03, aubrun05} that $\var(\lambda_1) \le Cn^{-1/3}$. Therefore above formula for the covariance and Theorem \ref{varconv} seem to imply that the proof is done. However, we have to be a little careful because Theorem \ref{varconv} works only for Gaussian fields on finite index sets. But this can be easily taken care of by taking the Gaussian field $(X_\bu)$ restricted to finer and finer nets of points in $\cc\pp^{n-1}$ and using the uniqueness of the maximizer and continuity to pass to the limit. 
\end{proof}

\section{Example: Multiple global maxima in the $NK$ fitness landscape}\label{nk}
Kauffman and Levin \cite{kauffmanlevin87} introduced a class of models for the evolution 
of hereditary systems, which has since become one of the most popular models in evolutionary biology and some other areas. They named it the $NK$ model because there are two parameters, $N$ and $K$. The model envisions a genome as consisting on $N$ genes, each of which exists as one of two possible alleles. The fitness score of an allele at a given site is determined by the alleles of the $K$ neighboring sites. Other than that, the fitnesses are as simple as possible, namely i.i.d., and the fitnesses of different sites are averaged to get the overall fitness of a genome. 

Let us define things formally. The space of all genomes is $\{0,1\}^N$. Let
\[
Y(i; \boe), \ 0\le i\le N-1, \ \boe\in \{0,1\}^{K+1}
\]
be a collection of i.i.d.\ random variables, assumed to be standard Gaussian for our purposes. Given $\bos \in \{0,1\}^N$, define the `fitness' of $\bos$ as
\[
F(\bos) := \sum_{i=0}^{N-1} Y(i; (\sigma_i,\sigma_{i+1},\ldots,\sigma_{i+K})),
\]
where the subscripts of $\sigma$ are wrapped around, i.e.\ $\sigma_{N + i} = \sigma_i$. The function $F:\{0,1\}^N \ra \rr$ is called the `fitness landscape', and the main objects of interest are the local and global maxima of this landscape. 

The $NK$ model has been extensively studied, but very little of it is rigorous. The first rigorous paper on the model, to the best of our knowledge, was due to Evans and Steinsaltz \cite{evanssteinsaltz02}. These authors used elegant arguments involving max-plus algebras to carry out computations about the global and local maxima of the $NK$ model when $K$ is fixed and $N\ra \infty$. One can also find in \cite{evanssteinsaltz02} a beautifully written introduction to the history and motivation behind the model. Further rigorous results were derived using different tools by Durrett and Limic \cite{durrettlimic03} who proved, among other things, a central limit theorem for the maximum fitness when $K$ is fixed and $N\ra \infty$. Some results about the local maxima in the case where $N$ and $K$ both tend to infinity were obtained by Limic and Pemantle~\cite{limicpemantle04}. 

Our goal is to show that when $N$ and $K$ both tend to infinity, the fitness landscape has many global near-maxima, which are `far away' from each other. In other words, there are many nearly globally fittest genomes that are drastically different from each other. To formalize this statement, we first need define what is meant by `far away' in the space of genomes. The $NK$ model naturally defines the following measure of proximity between two genomes $\bos$ and $\bos'$:
\begin{equation}\label{pnk}
p_{N,K}(\bos,\bos') := |\{i: (\sigma_i,\ldots, \sigma_{i+K})= (\sigma'_i,\ldots,\sigma'_{i+K})\}|. 
\end{equation}
This is  a natural definition because 
\[
p_{N,K}(\bos, \bos') = \cov(F(\bos),F(\bos')).
\]
Note also that if $K$ is large, then $\bos$ and $\bos'$ may be far apart even if the Hamming distance between $\bos$ and $\bos'$ is relatively small.

To understand the nature of the global maximum, we first have to have an idea about its size. It was shown by Evans and Steinsaltz \cite{evanssteinsaltz02} that the size of global maximum grows linearly in $N$ when $K$ is fixed (in an asymptotic sense). Following their argument, one can further deduce the surprising fact that when divided by $N$ the expected size of the global maximum can be bounded above and below by universal constants that do not depend on $K$.
\begin{lmm}\label{nkexp}
Irrespective of the value of $K$, we have
\[
\frac{N}{\sqrt{\pi}} \le \ee(\max_\bos F(\bos)) \le N \sqrt{2\log 2}.
\]
\end{lmm}
\noindent {\it Remark.} In fact, the bounds are sharp. The lower bound is achieved when $K=0$, and the upper bound is achieved (asymptotically) when $K=N-1$. Moreover, as we will see in the proof, the expected value is an increasing function of $K$. 
\begin{proof}
The upper bound is straightforward from Lemma \ref{max}. For the lower bound, first observe that if $G(\bos)$ is another measure of fitness, corresponding  to the $NK$ model with $K=0$, then for every $\bos, \bos'$, we have 
\[
\cov(G(\bos),G(\bos')) \ge \cov(F(\bos),F(\bos')),
\]
since it is quite clear that $p_{N,K}(\bos,\bos')$ is a decreasing function of $K$. Moreover $\var(F(\bos))=\var(G(\bos))= N$ for every $\bos$. Therefore by Slepian's lemma, 
\[
\ee(\max_\bos F(\bos)) \ge \ee(\max_\bos G(\bos)). 
\]
Now, if $(Z(i,\eta))_{0\le i\le N-1, \ \eta\in \{0,1\}}$ are the i.i.d.\ Gaussian fitness scores used to define $G$, then 
\[
\max_\bos G(\bos) = \sum_{i=0}^{N-1} \max\{Z(i,0),  Z(i,1)\}.
\]
It is easy to compute that the expected value of the maximum of two independent standard Gaussian random variables is $\pi^{-1/2}$. This completes the proof.
\end{proof}

The following theorem is the main result of this section. It shows that when $N$ and $K$ both tend to infinity, the fitness landscape exhibits the multiple peaks property. We first establish superconcentration using hypercontractivity, and then use Theorem \ref{p2p3} to deduce the existence of multiple peaks. 
\begin{thm}
Suppose $N > K \ge 1$. Then $\var(\max_\bos F(\bos)) \le CN/K$, where $C$ is a universal constant. As a consequence, there is another universal constant $C'$ such that for any $a > 1$, with probability at least 
\[
1-C'a^{-1/4}\sqrt{\log a}
\]  
there exists $A\subseteq \{0,1\}^N$ such that
\begin{enumerate}
\item[(a)] $|A| \ge a^{1/12}$,
\item[(b)] $p_{N,K}(\bos,\bos') < a^{-1/4} N$ for all $\bos,\bos'\in A$, $\bos\ne \bos'$, where $p_{N,K}$ is the measure \eqref{pnk} of proximity between genomes, and
\item[(c)] for all $\bos\in A$, 
\[
F(\bos) \ge \max_{\bos'\in \{0,1\}^N}  F(\bos') -  \frac{a N}{K}.
\]
\end{enumerate}
\end{thm}
\noindent{\it Remarks.} By Lemma \ref{nkexp} and the concentration of $\max F(\bos)$, we know that $\max F(\bos)$ is of order $N$. This shows that whenever $K$ is large, there are many near-maximal configurations that are `far apart' in the sense of the proximity measure $p_{N,K}$, since $p_{N,K}$ ranges between $0$ and $N$. The quantification of `many' and `far apart' depends on our choice of $a$. The theorem shows that any large $a$ works, as long as $a \ll K$. Of course, the theorem in its present form does not have any relevance for realistic values of $N$ and $K$, but we believe that there is room for improvement to an extent that can be of practical significance. Finally, note that the theorem proves strong multiple peaks when $K$ grows faster than $\sqrt{N}$. 

\begin{proof}
Let $M = \max_\bos F(\bos)$.
Let $\hat{\bos}$ be the maximizing configuration. With obvious notation, we have
\[
\fpar{M}{Y(i; \boe)} =  \ii_{\{(\hat{\sigma}_i,\ldots,\hat{\sigma}_{i+K}) = (\eta_1,\ldots,\eta_{K+1})\}}.
\]
Therefore by Lemma \ref{intparts}, Lemma \ref{hyper}, and symmetry, we have
\begin{align*}
&\var(M) \\
&= \int_0^\infty e^{-t}\sum_{(i,\boe)} \ee\bigl(\ii_{\{(\hat{\sigma}_i,\ldots,\hat{\sigma}_{i+K}) = (\eta_1,\ldots,\eta_{K+1})\}}\ii_{\{(\hat{\sigma}_i^t,\ldots,\hat{\sigma}_{i+K}^t) = (\eta_1,\ldots,\eta_{K+1})\}}\bigr) dt\\
&\le \int_0^\infty e^{-t} \sum_{(i,\boe)} \bigl[\pp((\hat{\sigma}_i,\ldots,\hat{\sigma}_{i+K}) = (\eta_1,\ldots,\eta_{K+1}))\bigr]^{1+\tanh(t/2)} dt\\
&= \int_0^\infty e^{-t} \sum_{(i,\boe)} 2^{-(K+1)(1+\tanh(t/2))} dt\\
&=\int_0^\infty e^{-t} N 2^{-(K+1)\tanh(t/2)} dt\le \frac{CN}{K},
\end{align*}
where $C$ is a universal constant. 
Now let $m = \ee(M)$ and $v = \var(M)$. 
From Lemma \ref{nkexp}, we know that $N/\sqrt{\pi} \le m\le N\sqrt{2\log 2}$, and from the above computation we know that $v\le CN/K$. Let $\sigma^2 = \max_\bos \var(F(\bos)) = N$. Take any $a > 1$, and let $l = [a^{1/12}] + 1$, $\epsilon = a^{-1/4} N$, and $\delta = a N/K$. Then 
\begin{align*}
\sqrt{\frac{vml^3\log l}{\delta \epsilon }}  + \biggl(\frac{v\sigma^4l^5 \log l}{\delta^3 m \epsilon}\biggr)^{1/4} &\le C a^{-1/4}\sqrt{\log a} + C\biggl(\frac{K^2a^{-7/3}\log a}{N^2}\biggr)^{1/4},
\end{align*}
where $C$ is a universal constant. The second term is dominated by the first. The result now follows from Theorem \ref{p2p3}.
\end{proof}

\section{Example: Chaos in directed polymers}\label{polymer}
We introduced directed polymers in Section \ref{intro}. Let us now refresh the reader's memory by defining it once again.

The (1+1)-dimensional directed polymer model in Gaussian environment is defined as follows. Let $\zz_+^2$ denote the set of all pairs of non-negative integers, with the lattice graph structure. Let $\bbg = (g_v)_{v\in \zz_+^2}$ be a collection of i.i.d.\ standard Gaussian random variables. A directed polymer of length $n$ is a sequence of $n$ adjacent points, beginning at the origin, such that each successive point is either to the right or above the previous point. Thus, there are a total of $2^{n-1}$ directed polymers of length $n$. Let $\mathcal{P}_n$ denote this set. An element of $\mathcal{P}_n$ will generally be denoted by $p=(v_0,\ldots,v_{n-1})$, where $v_0=0$. The energy of a polymer $p = (v_0,\ldots,v_{n-1})$ in the environment $\bbg$ is defined as
\[
E(p) := -\sum_{i=0}^{n-1} g_{v_i}.
\] 
For each $p\in \mathcal{P}_n$, let $X_p := -E(p)$ denote the `weight' of the path, and let $\bbx = (X_p)_{p\in \mathcal{P}_n}$. Our object of interest is the polymer with the minimum energy (i.e.\ maximum weight), usually called the ground state of the system. Suppressing the subscript $n$, we simply denote the minimum energy path by $P$ and its energy by $-M$, that is,
\[
M := \max_{p\in \mathcal{P}_n}X_p = X_P.
\]
Note that for any $p,p'$,
\[
\cov(X_p, X_{p'}) = |p\cap p'|,
\]
where $|p\cap p'|$ denotes the number of vertices common to $p$ and $p'$.

If $\bbg'$ is an independent copy of $\bbg$, and we define the perturbed environment $\bbg^t := \est \bbg + \esst \bbg'$, then the energies of the paths defined in the environment $\bbg^t$ correspond to our usual definition of $\bbx^t$. Let $P^t$ denote the maximum weight path in the environment $\bbg^t$ (henceforth, simply the `maximal path'). The main result of this section is the following.
\begin{thm}\label{poly}
For some universal constant $C$, we have
\[
\var(M)\le \frac{Cn}{\log n}.
\]
Consequently, for any $t\ge 0$,
\[
\ee|P\cap P^t| \le \frac{Cn}{(1-e^{-t})\log n}. 
\]
\end{thm}
\noindent It is easy to see how Theorem \ref{polyad} follows from this result. Another remark is that the superconcentration of the ground state energy has implications about the fluctuations of the polymer shape and specially the end point of the minimum energy polymer; see Wehr and Aizenman  \cite{aizenmanwehr90} for further insights.

One may object that the perturbation described above is not a small perturbation at all, because we are perturbing all coordinates, and moreover  independently. Indeed, the regular notion of perturbation in noise-sensitivity theory involves choosing a small fraction of coordinates at random and replacing the weights with independent copies. 
However, the two notions lead to the same results in practice, because in one case  we are giving large perturbations to a small collection of coordinates, while in the other case (i.e.\ our case) we are giving small perturbations to all coordinates. One can verify that with correct calibration, the two perturbations have similar effects on almost every conceivable summary statistic. 

Secondly, our definition of perturbation of a Gaussian field is the one favored by the physicists, who see it as an Ornstein-Uhlenbeck flow over a small period of time.

The chaos property for directed polymers was first argued heuristically by Huse, Henley, and Fisher \cite{husehenleyfisher85}. It was subsequently studied numerically and theoretically by Zhang~\cite{zhang87} and M\'ezard \cite{mezard90}. A detailed theoretical study was done by Fisher and Huse \cite{fisherhuse91} and later, another one by Hwa and Fisher~\cite{hwafisher94}. For a more recent work, see da Silveira and Bouchaud \cite{silveirabouchaud04}. 
However, none of these papers give rigorous proofs, or even proofs that can be made rigorous with existing technology. To the best of our knowledge, no such proof exists.

The scheme of the proof of Theorem \ref{poly} is straightforward given the tools we have. We apply Talagrand's technique of bounding variances using hypercontractivity, and use the resulting superconcentration bound to infer chaos via Theorem \ref{main} (and the actual bound via Theorem \ref{varconv}). However, on our way to applying hypercontractivity, we run into the difficulty that the probability of the optimal path passing through a given vertex is completely unknown; even a reasonable upper bound is unknown. This problem is overcome by a brilliant trick from a paper of Benjamini, Kalai, and Schramm~\cite{bks03}, where a similar variance-bounding problem for first passage percolation was tackled. But serious technical issues remain even after this, because adapting the BKS trick requires a lot more effort in this case compared to the case handled in~\cite{bks03}, because of the restriction that the paths have to be directed. In fact, this is the sole reason why the proof in its present form does not generalize to dimensions higher than $2$, even though the BKS result for first passage percolation holds in all dimensions.

We should mention here that there exists a large  amount of rigorous mathematics on directed polymers, even if there is none on the chaos problem. The rigorous study was probably initiated by Imbrie and Spencer \cite{imbriespencer88}, and subsequently carried forward by many authors (e.g.\ \cite{bolthausen89},\cite{sinai95}, \cite{albeveriozhou96}, \cite{songzhou96}, \cite{kifer97}, \cite{carmonahu02}, \cite{carmonahu04}, \cite{cometsyoshida06}, \cite{cometsshigayoshida04}). A very notable contribution was due to Johansson \cite{johansson00}, who found a miraculous way to do exact computations in the model with Geometric vertex weights instead of Gaussian, showing that $M$ has fluctuations of order $n^{1/3}$, and moreover a Tracy-Widom limiting distribution upon proper centering and scaling. The result was extended to binary edge weights in Johansson \cite{johansson01}, Section 5.1 (see also \cite{gravnertracywidom01}). The recent work of Cator and Groeneboom~\cite{catorgroeneboom06} implies a probabilistic proof of the  $n^{1/3}$ fluctuations, but again for Geometric vertex weights.

We need some preparation before embarking on the proof of Theorem \ref{poly}. First of all, let us remind the reader of our convention that $C$ denotes any positive universal constant whose value may change from line to line. This convention will be repeatedly invoked in the remainder of this section. Next, for each $w \in \zz_+^2$, define the translated environment $\bbg_w = (g_{w,v})_{v\in \zz_+^2}$ as $g_{w,v} := g_{w+v}$.
Now fix $n$, and define $\bbx_w$, $\bbx_w^t$, $P_w$, $P_w^t$, and $M_w$ as the analogs of $\bbx$, $\bbx^t$, $P$, $P^t$, and $M$ for the environment $\bbg_w$. Next, let
\begin{equation}\label{bdef}
B := \{(i,j): 1\le i,j\le [n^{1/4}]\},
\end{equation}
and let 
\[
\overline{M} := \frac{1}{|B|} \sum_{w\in B} M_w.
\]
Our first lemma is the following.
\begin{lmm}\label{mbar}
We have 
\[
\var(\overline{M}) \le \frac{Cn}{\log n}.
\]
\end{lmm}
\begin{proof}
With obvious notation, we have
\[
\fpar{\overline{M}}{g_v} = \frac{1}{|B|} \sum_{w\in B} \fpar{M_w}{g_v} = \frac{1}{|B|} \sum_{w\in B} \ii_{\{v-w\in P_w\}}
\]
Therefore by Lemma \ref{intparts} and Lemma \ref{hyper} we have
\begin{align*}
\var(\overline{M}) &= \int_0^\infty e^{-t} \sum_{v\in \zz_+^2} \ee\biggl[\biggl(\frac{1}{|B|}\sum_{w\in B} \ii_{\{v-w\in P_w\}}\biggr)\biggl(\frac{1}{|B|}\sum_{w\in B} \ii_{\{v-w\in P_w^t\}}\biggr)\biggr] dt\\
&\le \int_0^\infty e^{-t} \sum_{v\in \zz_+^2} \biggl[\ee\biggl(\frac{1}{|B|}\sum_{w\in B} \ii_{\{v-w\in P_w\}}\biggr)\biggr]^{1+\tanh(t/2)} dt.
\end{align*}
Now fix any $v\in \zz_+^2$. Then
\begin{align*}
\ee\biggl(\frac{1}{|B|}\sum_{w\in B} \ii_{\{v-w\in P_w\}}\biggr) &= \frac{1}{|B|}\sum_{w\in B} \pp(v-w\in P_w)\\
&= \frac{1}{|B|}\sum_{w\in B} \pp(v-w\in P)= \ee\biggl(\frac{1}{|B|}\sum_{w\in B} \ii_{\{v-w\in P\}}\biggr).
\end{align*}
Since $P$ is an directed path and $B$ is an $[n^{1/4}]\times [n^{1/4}]$ square, we certainly have
\[
\sum_{w\in B} \ii_{\{v-w\in P\}} \le 2[n^{1/4}] - 1.
\]
Combining the last three observations, we have that
\begin{align*}
\var(\overline{M}) &\le \int_0^\infty Ce^{-t}n^{-\frac{1}{4}\tanh(t/2)} \ee\biggl(\frac{1}{|B|} \sum_{v\in \zz_+^2}\sum_{w\in B} \ii_{\{v-w\in P\}}\biggr) dt \\
&= \int_0^\infty Ce^{-t}n^{-\frac{1}{4}\tanh(t/2)} \ee\biggl(\frac{1}{|B|} \sum_{w\in B} \sum_{v\in \zz_+^2} \ii_{\{v-w\in P\}}\biggr) dt \\
&= \int_0^\infty Ce^{-t}n^{1-\frac{1}{4}\tanh(t/2)} dt \le \frac{Cn}{\log n}.
\end{align*}
This completes the proof of the lemma.
\end{proof}
Let us now introduce some further notation. For any $v\in \zz_+^2$, let $|v|$ denote the sum of its coordinates. For any $u,v\in \zz_+^2$, we write $u\le v$ if $v-u \in \zz_+^2$. For any such $u, v$, let 
\[
M_{u\ra v} := \max\biggl\{\sum_{u\in p\backslash\{v\}} g_u: \text{ $p$ is a directed path from $u$ to $v$}\biggr\}.
\]
For any $u\in \zz_+^2$, $n\ge 1$, let
\[
M_{u}^n := \max\biggl\{\sum_{u\in p} g_u: \text{ $p$ is a directed path of length $n$, starting at $u$}\biggr\}.
\]
Let $P_{u\ra v}$ and $P_u^n$ denote the maximizing paths.

Clearly, the distribution of $M_{u}^n$ depends only on $n$, and the distribution of $M_{u\ra v}$ depends only on $v-u$. Let us now prove some inequalities for the expected values of these quantities. In the following $0$ can either denote the real number $0$, or the point $(0,0)$ in $\zz^2$, depending on the context. 
\begin{lmm}\label{mdiff}
For any $1\le m\le n$, we have
\[
\ee(M_0^n) \ge \ee(M_0^m) + C(n-m).
\]
\end{lmm}
\begin{proof}
Clearly, it suffices to prove when $n = m+1$. Let $(v_0,\ldots, v_{m-1})$ denote the maximal path of length $m$. Let $u$ and $w$ be the vertices immediately to the right and above $v_{m-1}$. Since the definition of the maximal path of length $m$ does not involve vertices $v$ with $|v| \ge m$, it follows that $(g_u,g_w)$ is still a pair of i.i.d.\ standard Gaussian random variables, even though $u$ and $w$ are random. To complete the proof, note that $M_0^{m+1} \ge M_0^m + \max\{g_u, g_w\}$.
\end{proof}
\begin{lmm}\label{mup}
For any $v = (v^x, v^y)\in \zz_+^2$, we have
\[
\ee(M_{0\ra v}) \le C |v|\sqrt{H\biggl(\frac{v^x}{v^x + v^y}\biggr)},
\]
where $H:[0,1]\ra \rr$ is the function $H(a)= -a\log a -(1-a)\log (1-a)$. 
\end{lmm}
\begin{proof}
Note that the total number of directed paths from $0$ to $v$ is
\[
{v^x + v^y\choose v^x}.
\] 
For any $a\in [0,1]$, we have
\[
1 \ge {v^x + v^y\choose v^x}a^{v^x}(1-a)^{v^y}.
\] 
Taking $a = v^x/(v^x + v^y)$, we see that
\[
\log {v^x + v^y\choose v^x}\le -v^x \log a - v^y \log (1-a) = H(a)|v|.
\] 
The result now follows by Lemma \ref{max}.
\end{proof}
\begin{lmm}\label{expbd}
Let $P$ denote the maximal path of length $n$. Let $v=(v^x,v^y)$ be a vertex with $|v|\le n-1$. There is a universal constant $c \in (0,1)$ such that if $\min\{v^x, v^y\} \le c|v|$, then
\[
\pp(v\in P)\le 2e^{-c|v|^2/n}. 
\]
\end{lmm}
\begin{proof}
Let $\widetilde{P}$ denote the directed path of length $n$, originating from $0$ and passing through $v$, that maximizes the sum of weights among all such paths. Note that this is just the concatenation of the paths $P_{0\ra v}$ and $P_{v}^{n-|v|}$. Let $\widetilde{M} = \sum_{v\in \widetilde{P}} g_v$. Then note that by Lemma  \ref{mdiff},
\begin{align*}
\ee(\widetilde{M}) &= \ee(M_{0\ra v}) + \ee(M_v^{n-|v|})\\
&= \ee(M_{0\ra v}) + \ee(M_0^{n-|v|})\\
&\le \ee(M_{0\ra v}) + \ee(M) - C |v|,
\end{align*}
where $C$ is the constant from Lemma \ref{mdiff}. For the rest of this proof, $C$ will denote this constant. 
From the above bound, we see that if $\ee(M_{0\ra v}) \le C|v|$, then by Proposition \ref{borell} we have
\begin{align*}
\pp(v\in P) &= \pp(\widetilde{P}= P) \le \pp(\widetilde{M} = M) \\
&\le \pp\biggl(\widetilde{M} \ge \ee(\widetilde{M}) + \frac{\ee(M)- \ee(\widetilde{M})}{2}\biggr) \\
&\qquad + \pp\biggl(M \le \ee(M) - \frac{\ee(M)- \ee(\widetilde{M})}{2}\biggr)\\
&\le 2 \exp\biggl(-\frac{(\ee(M)-\ee(\widetilde{M}))^2}{8n}\biggr)\\
&\le 2 \exp\biggl(-\frac{(C|v|-\ee(M_{0\ra v}))^2}{8n}\biggr).
\end{align*}
Now, the function $H$ in Lemma \ref{mup} satisfies $H(a)=H(1-a)$ and 
\[
\lim_{a\ra 0}H(a)=\lim_{a\ra 1}H(a) = 0.
\]
Therefore, there exists a constant $c \in (0,1)$ such that if $\min\{v^x, v^y\} \le c|v|$, then $\ee(M_{0\ra v})\le C|v|/2$, where $C$ is the universal constant in Lemma \ref{mdiff}. This proves the current lemma by choosing $c$ sufficiently small. 
\end{proof}
\begin{lmm}\label{lbd}
Denote the vertices of the maximal path $P$ by $(v_0,\ldots,v_{n-1})$. Extend the path $P$ to an infinite directed path $P'$ by adding a sequence of points $v_n,v_{n+1},\ldots$, where each $v_i$ is the point immediately to the right of $v_{i-1}$ if $i$ is odd, and immediately above $v_{i-1}$ if $i$ is even. Take any $w\in \zz_+^2\backslash\{0\}$ with $|w|\le c(n-1)$, where $c$ is the constant from Lemma \ref{expbd}. Let $L := \min\{i: v_i\ge w\}$.
Then for any $l\le n-1$ such that $|w|\le cl$, we have
\[
\pp(L=l)\le 2 e^{-cl^2/n}.
\]
We also have $\pp(L\ge n) \le 2|w|e^{-c(n-1)^2/n}$. As a consequence of these bounds, we have $\ee(L) \le C(|w| + \sqrt{n})$ for some universal constant $C$.
\end{lmm}
\begin{proof}
Let $k := |w|$. It is easy to see that the definition of $L$ ensures that 
\begin{equation}\label{lbds}
k \le L\le n + k - 1.
\end{equation}
Take any integer $l\in [k,n-1]$. Since $|v_L|=L$, we have 
\begin{align*}
\pp(L = l) &= \pp(v_L = (w^x, l-w^x)  \text{ or } v_L = (l-w^y, w^y))\\
&\le \pp((w^x, l-w^x) \in P) + \pp((l-w^y, w^y)\in P). 
\end{align*}
The proof of the first bound can now be completed by applying Lemma \ref{expbd}. For the second, note that again by Lemma \ref{expbd},
\begin{align*}
\pp(L\ge n) &\le \sum_{i=0}^{w^x-1}\pp((i,n-1-i) \in P) + \sum_{i=0}^{w^y-1} \pp((n-1-i, i)\in P) \\
&\le 2|w|e^{-c(n-1)^2/n}. 
\end{align*}
Together with \eqref{lbds}, this completes the proof. 
\end{proof}
\begin{lmm}\label{rnklemma}
Let $\mathcal{R}_{n,k}$ be the set of all directed paths whose end points $u = (u^x, u^y)$ and $v=(v^x, v^y)$ satisfy $|u|\le 2n$, $|v|\le 2n$, and 
\begin{equation}\label{rnk}
\min\{|u^x-v^x|, |u^y- v^y|\} \le k.
\end{equation}
Let 
\[
R_{n,k} := \max_{p\in \mathcal{R}_{n,k}}\frac{|\sum_{v\in p} g_v|}{\sqrt{|p|}},
\]
where $|p|$ denotes the number of vertices in $p$. Then $\ee(R_{n,k}^2)\le C k\log n$. 
\end{lmm}
\begin{proof}
We wish to use Lemma \ref{max}. Clearly, for each $p\in \mathcal{R}_{n,k}$, 
\[
\frac{\sum_{v\in p} g_v}{\sqrt{|p|}} \sim N(0,1).
\]
Let us now find an upper bound on the size of $\mathcal{R}_{n,k}$. The end points of a path can be chosen in at most $Cn^2$ ways. Given the end points, in view of the restriction \eqref{rnk} and the directed nature of the paths, we can employ the same counting argument as in the proof of Lemma \ref{mup} to conclude that the number of paths is $\le (Cn)^k$. Thus, $|\mathcal{R}_{n,k}|\le n^2 (Cn)^k$. The claim now follows by Lemma \ref{max}. 
\end{proof}
\begin{lmm}\label{mmwbd}
For any $w\in \zz_+^2$, we have
\[
\ee(M-M_w)^2 \le C( |w|^2 + |w|\sqrt{n}) \log n.
\]
\end{lmm}
\begin{proof}
As before, let $k=|w|$. Since $\ee(M^2)$ is bounded by $Cn^2$ (easily verified by Lemma \ref{max}), we can assume without loss of generality that $k\le c(n-1)$. Let $P'$ and $L$ be defined as in Lemma \ref{lbd}. 
Denote the components of $v_i$ by $v_i^x$ and $v_i^y$. Then we know that either $v_L^x = w^x$ or $v_L^y = w^y$. Define a directed path $P'' = (u_0,\ldots,u_{n-1})$ as follows.  If $v_L^x = w^x$, let $u_0 = w$ and let $u_i$ be the point immediately above $u_{i-1}$ for $1\le i\le L-k$. If $v_L^x > w^x$ and $v_L^y = w^y$, let $u_0 = w$ and let $u_i$ be the point immediately to the right of $u_{i-1}$ for $1\le i\le L-k$. Note that in either case, we end up with $u_{L-k} = v_L$ because $|v_L|=L$. Thereafter, merge the path with $P'$, that is, let
\[
u_i = v_{k + i} \ \text{ for } \ L-k+1\le i\le n-1.
\]
Let $\mathcal{R}_{n,k}$ and $R_{n,k}$ be defined as  in Lemma \ref{rnklemma} above. Clearly, the paths $(v_0,\ldots, v_{\min\{L,n-1\}})$ and $(u_0,\ldots, u_{L-k})$ belong to $\mathcal{R}_{n,k}$. With this observation, and the fact that $P''$ is a directed path of length $n$ starting at $w$, we have
\begin{align*}
M_w &\ge \sum_{i=0}^{n-1} g_{u_i} = \sum_{i=0}^{L-k} g_{u_i} + \sum_{i=L}^{n+k-1} g_{v_i} \\
&= M - \sum_{i=0}^{\min\{L, n-1\}} g_{v_i} + \sum_{i=\max\{n, L\}}^{n+k-1} g_{v_i} + \sum_{i=0}^{L-k} g_{u_i}\\
&\ge M - 2\sqrt{L}R_{n,k} - k \max_{|v|\le 2n} |g_v|.
\end{align*}
Therefore, by the Cauchy-Schwarz inequality, Lemma \ref{max}, Lemma \ref{lbd}, and Lemma \ref{rnklemma}, it follows that
\begin{equation}\label{mmw1}
\ee(M-M_w)_+^2 \le C (|w|^2 + \sqrt{n}|w|) \log n.
\end{equation}
Next, let us get a bound on $\ee(M_w-M)_+^2$. Denote the vertices of $P_w$ by $(v_{w,0},\ldots,v_{w,n-1})$. Let $(u_0,\ldots,u_{k-1})$ be any fixed directed path with $u_0=0$ and $u_{k -1}= w$. Let $P_w'$ be the directed path 
\[
(u_0,\ldots,u_{k-1}, w+v_{w,1}, w+v_{w,2},\ldots, w+v_{w,n-k}).
\] 
Then note that
\begin{align*}
M &\ge \sum_{v\in P_w'} g_v = M_w + g_{u_0}+\cdots + g_{u_{k-2}} - g_{w+v_{w,n-k+1}}-\cdots - g_{w+v_{w,n-1}}\\
&\ge M_w - 2(k-1) \max_{|v|\le 2n} |g_v|.
\end{align*}
This shows that
\begin{equation}\label{mmw2}
\ee(M_w - M)_+^2 \le C |w|^2\log n.
\end{equation}
Combining \eqref{mmw1} and \eqref{mmw2}, our proof is done.
\end{proof}
\begin{proof}[Proof of Theorem \ref{poly}]
For each $w\in B$, we have
\[
\ee(M-M_w)^2 \le C (|w|^2 + \sqrt{n}|w|) \log n\le C n^{3/4}\log n.
\]
Therefore,
\[
\ee(M-\overline{M})^2 \le Cn^{3/4}\log n.
\]
Thus, by Lemmas \ref{mbar} and \ref{mmwbd} we have
\begin{align*}
\var(M) &\le \ee(M- \ee(\overline{M}))^2 \\
&\le 2\ee(M-\overline{M})^2 + 2\var(\overline{M})\\
&\le Cn^{3/4}\log n + \frac{Cn}{\log n}.
\end{align*}
This completes the proof.
\end{proof}

\section{Example: Generalized SK model of spin glasses}\label{skmodel}
Let $n$ be a positive integer and let $\Sigma_n = \{-1,1\}^n$. Suppose we have $n$ magnetic particles, each having spin $+1$ or $-1$. A typical element of $\bos = (\sigma_1,\ldots,\sigma_n)\in \Sigma_n$ is called a `configuration' of spins. Let $g = (g_{ij})_{1\le i<j\le n}$ be a collection of independent standard Gaussian random variables. Given a realization of $g$, define the energy of a configuration $\bos$ as
\[
H_n(\bos) = - \frac{1}{\sqrt{n}} \sum_{1\le i<j\le n} g_{ij} \sigma_i \sigma_j.
\]
The energy landscape so defined corresponds to the famous Sherrington-Kirkpatrick model of spin glasses \cite{sk75} in the absence of an external field. In the last thirty years, the SK model has been an inspiration for a large body of groundbreaking physics (see the M\'ezard-Parisi-Virasoro book  \cite{mezardetal87}) as well as beautiful rigorous mathematics (see e.g.\ \cite{alr87}, \cite{frohlichzegarlinski87}, \cite{cometsneveu95}, \cite{shcherbina97}, \cite{talagrand98}, \cite{guerra02}, \cite{guerra03}, \cite{tindel05}, \cite{talagrand06}). An extensive collection of rigorous results can be found in Talagrand's book \cite{talagrand03}, a new edition of which is in preparation.

A natural variant of the SK model is the $p$-spin SK model, where the energy of a configuration is defined as
\begin{equation}\label{pspin}
H_{n,p}(\bos) = -\frac{1}{n^{(p-1)/2}}\sum_{1\le i_1,i_2,\ldots, i_p\le n} g_{i_1i_2\cdots i_p} \sigma_{i_1}\sigma_{i_2}\cdots\sigma_{i_p},
\end{equation}
where again, $(g_{i_1i_2\cdots i_p})_{1\le i_1,\ldots,i_p\le n}$ is a fixed realization of i.i.d.\ standard Gaussian random variables. (Usually, the sum is taken over distinct $i_1,\ldots,i_p$. We take it over all $i_1,\ldots,i_p$ to avoid certain technical inconveniences.) The $p$-spin model was suggested by Derrida, and subsequently studied by Gross and M\'ezard \cite{grossmezard84} and Gardner \cite{gardner85}. 

A generalized version of the SK model that covers all $p$-spin models was considered by Talagrand in his celebrated paper on the Parisi formula \cite{talagrand06}. It is simply a linear combination of the $p$-spin energies over all $p$ and covers all cases considered till now. Given a sequence of non-negative real numbers $\mathbf{c} = (c_2,c_3,\ldots)$ such that 
\[
\sum_{p=2}^\infty c_p =1, 
\]
define the energy function 
\[
H_{n,\mathbf{c}}(\bos) := \sum_{p=2}^\infty c_p^{1/2} H_{n,p}(\bos),
\]
where $H_{n,p}$ is the $p$-spin energy defined above in \eqref{pspin}. Then the usual SK model corresponds to the sequence $(1,0,0,\ldots)$, and the $p$-spin model corresponds to the sequence that has $1$ at the $p$th position and $0$ elsewhere.

The objective of this section is to analyze the energy landscape of the generalized SK model. In particular, we are interested in the behavior of the ground state, i.e.\ the configuration with minimum energy, and the fluctuations of the energy of the ground state. 

There are two notable conjectures about the ground state of the SK model, both seemingly beyond the reach of current technology. One is that the ground state exhibits chaos (exactly in our sense). A physics proof for chaos was given by McKay, Berger, and Kirkpatrick \cite{mckayetal82}. The second conjecture is about the `fluctuation exponent' of the ground state energy. 
\begin{defn}
The energy of the ground state is said to have fluctuation exponent $\rho$ if it has fluctuations of order $n^{\rho +o(1)}$. 
\end{defn}
\noindent Classical results like Proposition~\ref{varcrude} and Proposition \ref{borell} imply that for the generalized SK model, $\rho \le 1/2$. It is predicted by physicists \cite{kondor83, crisantietal92} that $\rho = 1/6$ in the SK model, although the prediction has not yet been reliably verified by simulations. 

As we know from our Theorem \ref{main}, the two problems are related. Indeed, superconcentration and hence chaos happens if the fluctuation exponent is anything strictly less than $1/2$. The main result of this section is that the fluctuation exponent of the ground state energy in the generalized SK model is at most $3/8$ if $c_p$ decreases to zero sufficiently slowly as $p\ra \infty$. 
\begin{thm}\label{exponent}
Let $I(x) := \frac{1}{2}((1+x)\log (1+x) + (1-x)\log (1-x))$, and $(c_p^*)_{p\ge 2}$ be constants such that for all $x\in (-1,1)$, 
\[
\frac{I(x)}{2\log 2 - I(x)} = \sum_{p=2}^\infty c_p^* x^{p}.
\] 
Then $c_p^*\ge 0$ for all $p$ and $\sum c_p^* = 1$. Suppose $\mathbf{c} = (c_p)_{p\ge 2}$ is any non-negative sequence such that $\sum c_p = 1$, and for all $r\ge 2$,
\begin{equation}\label{major1}
\sum_{p=2}^r c_p \le \sum_{p=2}^r c_p^*.
\end{equation}
Then the ground state energy of the generalized SK model defined by the sequence $\mathbf{c}$ has fluctuation exponent $\le 3/8$, and consequently, the ground state is chaotic. 
\end{thm}
\noindent Incidentally, it was shown by Wehr and Aizenman \cite{aizenmanwehr90} that the lattice spin glass (i.e.\ the Edwards-Anderson model) is not superconcentrated, and hence, not chaotic in our sense. Therefore, if we are looking for chaos in spin glasses, the only option is to look in mean-field models. 

The proof of Theorem \ref{exponent} is based on our result about extremal Gaussian fields, namely, Theorem \ref{extreme}. The minorizing condition \eqref{major1} suffices to guarantee extremality of the energy landscape considered as a Gaussian field.  We will actually prove a more general version of Theorem \ref{exponent}, with precise quantitative bounds. 
Fix $n$, and consider a centered Gaussian field $\bbx = (X_{\bos})_{\bos \in \{-1,1\}^n}$ satisfying
\[
\cov(X_{\bos}, X_{\bos'}) = \xi\biggl(\frac{\bos \cdot \bos'}{n}\biggr) \ \text{ for all } \bos \in \{-1,1\}^n,
\] 
where $\bos \cdot \bos' = \sum_{i=1}^n \sigma_i \sigma_i'$ is the usual inner product, and $\xi:[-1,1]\ra [-1,1]$ is a function with $\xi(1)=1$. Let $\bbx'$ be an independent copy of $\bbx$. As usual, define the perturbed field
\[
\bbx^t = \est \bbx + \esst \bbx'.
\]
Let $\hat{\bos}^t = \argmax_\bos X^t_\bos$. The following theorem is a quantitative version of Theorem \ref{exponent}.
\begin{thm}\label{spinchaos}
Let $I(x)$ be as in Theorem \ref{exponent}. Suppose that
\begin{equation}\label{major2}
|\xi(x)|\le \xi(|x|) \ \text{ and } \ \xi(x) \le \frac{I(x)}{2\log 2 - I(x)} \ \ \text{for all } \ x\in (-1,1).
\end{equation}
Then we have the bounds
\begin{align*}
\var(\max_\bos X_\bos) &\le C\biggl(\frac{\log n}{n}\biggr)^{1/4},  \ \ \text{and} \\ 
 \ee\biggl(\xi\biggl(\frac{\hat{\bos}^0\cdot\hat{\bos}^t}{n}\biggr)\biggr)  &\le \frac{C}{\esst}\biggl(\frac{\log n}{n }\biggr)^{1/4},
\end{align*}
where $C$ is a universal constant. 
\end{thm}
\noindent Before proving Theorem \ref{spinchaos}, let us first show that it implies Theorem \ref{exponent}. First of all, if we define $X_\bos =  - n^{-1/2}H_{n,\mathbf{c}}(\bos)$, then
\[
\cov(X_\bos, X_{\bos'}) = \sum_{p=2}^\infty c_p \biggl(\frac{\bos\cdot\bos'}{n}\biggr)^p.
\]
Thus, we are in the setting of Theorem \ref{spinchaos} with $\xi(x) = \sum c_p x^p$. So if we can show that \eqref{major2} follows from \eqref{major1}, then Theorem \ref{spinchaos} would imply that $\var(\max_\bos X_\bos) \le n^{-1/4 + o(1)}$ and hence that 
\[
\var(\min_\bos H_{n,\mathbf{c}}(\bos)) \le n^{3/4 + o(1)},
\]
which proves the claim. The implication of \eqref{major2} from \eqref{major1} is proved as follows. First, it is easy  to verify that the power series for $I(x)$ has non-negative coefficients, and therefore so does
\[
\frac{I(x)}{2\log 2 - I(x)} = \sum_{k=1}^\infty \biggl(\frac{I(x)}{2\log 2}\biggr)^k.
\] 
For each $r$, let $C_r = \sum_{p=2}^r c_p$,  and $C^*_r = \sum_{p=2}^r c^*_p$, with $C_1=C^*_1=0$. The assumption \eqref{major1} says that $C_r \le C^*_r$ for each~$r$. Thus for any $x\in (0,1)$,
\begin{align*}
\xi(x) = \sum_{p=2}^\infty c_p x^p &= \sum_{p=2}^\infty (C_p - C_{p-1}) x^p\\
&= \sum_{p=2}^\infty (x^p - x^{p+1}) C_p\\
&\le \sum_{p=2}^\infty (x^p-x^{p+1}) C^*_p = \sum_{p=2}^\infty c_p^* x^p = \frac{I(x)}{2\log 2- I(x)}.
\end{align*}
Since $|\xi(x)|\le \xi(|x|)$ and $I$ is symmetric, the inequality holds for $x\in (-1,0]$ as well. This completes the argument for Theorem \ref{exponent}.
\begin{proof}[Proof of Theorem \ref{spinchaos}]
We use Theorem \ref{gencorr}. 
Let ${\bf 1}$ denote the configuration of all $1$'s. Then by symmetry, we have
\begin{align*}
\sum_{\bos,\bos'\in \{-1,1\}^n}2^{-2n/(1+\xi(\frac{\bos\cdot\bos'}{n}))}  &= 2^n\sum_{\bos\in \{-1,1\}^n}\exp\biggl(-\frac{2n\log 2}{1+\xi\bigl(\frac{{\bf 1}\cdot\bos}{n}\bigr)}\biggr).
\end{align*}
By the binomial theorem, we know that the number of configurations that have $\sum_{i=1}^n \sigma_i = k$ is exactly
\[
{n \choose \frac{n+k}{2}},
\]
which is interpreted as zero if $k$ and $n$ have different parity. Now, we have that for any $p\in [0,1]$,
\[
{n \choose \frac{n+k}{2}} p^{(n+k)/2}(1-p)^{(n-k)/2} \le 1.
\]
Taking $p = (n+k)/2n$, we get 
\[
{n \choose \frac{n+k}{2}}\le  e^{n(\log 2 - I(k/n))},
\]
where $I(x)$ is defined in the statement of the theorem. Again, the hypothesis implies that
\[
\frac{2\log 2}{1+\xi(x)} \ge 2\log 2 - I(x)\ \ \text{for all } \ x\in [-1,1],
\]
Thus, we have
\begin{align*}
\sum_{\bos,\bos'\in \{-1,1\}^n}2^{-2n/(1+\xi(\frac{\bos\cdot\bos'}{n}))}  &\le  2^n\sum_{k=-n}^n\exp\biggl(-\frac{2n\log 2}{1+\xi(k/n)}\biggr) {n \choose \frac{n+k}{2}}\\
&\le C n. 
\end{align*}
The proof now follows from Theorem \ref{gencorr}.
\end{proof}

\section{Example: The Discrete Gaussian Free Field}\label{dgff}
In this section, we show that the Discrete Gaussian Free Field (DGFF) on an $n\times n$ grid (defined below) is a superconcentrated Gaussian field. The massless Gaussian free field is an important mathematical object, inspiring a substantial amount of rigorous literature. It is essentially a higher dimensional analog of Brownian motion, where the dimension of time (rather than space) is higher than one. Although initially introduced as a toy model for the Ising interface, the topic has grown in its own right and has found important intersections with subjects as diverse as quantum gravity and stochastic Loewner evolutions. The DGFF is a finite approximation to the massless free field, just as random walk is a finite approximation of  Brownian motion. For further motivation, definitions, and a review of the rigorous literature, we refer to Giacomin \cite{giacomin00} and the excellent recent survey of Sheffield \cite{sheffield07}.  


\subsection{Zero boundary condition} Let $V_n := \{0, \ldots, n-1\}^2$, and $\partial V_n$ be the inner boundary, that is, the points in $V_n$ which have a nearest neighbor outside. Let int$(V_n) := V_n\backslash \partial V_n$. The two-dimensional discrete Gaussian free field on $V_n$ with zero boundary condition  is defined as a family $\Phi_n = \{\phi_x\}_{x\in V_n}$ of centered gaussian random variables with covariances given by the discrete Green's function of the 
(discrete) Laplacian on int$(V_n)$. This means, explicitly, that  $\phi_x \equiv 0$ for $x\in \partial V_n$, and 
\begin{equation}\label{greencov}
\cov(\phi_x, \phi_y) = G_n(x,y) = \ee_x\biggl(\sum_{i=0}^{\tau_{\partial V_n}} \ii_{\{\eta_i = y\}}\biggr), \ \ \ x,y\in \mathrm{int}(V_n),
\end{equation}
where $\{\eta_i\}_{i\ge 0}$ is a standard symmetric nearest neighbor random walk on the 
two-dimensional lattice $\mathbb{Z}^2$, starting at $x$, and $\tau_{\partial V_n}$ is the first entrance time in $\partial V_n$. The law of $\Phi_n$ is the Gaussian distribution with density function proportional to
\begin{equation}\label{ffdens}
\exp\biggl( - \frac{1}{8}\sum_{x\sim y} (\phi_x - \phi_y)^2\biggr),
\end{equation}
where $x\sim y$ means $x$ and $y$ are neighbors in $V_n$ (each pair counted once), with the understanding that we set $\phi_x \equiv 0$ for $x\in \partial V_n$ in the above formula. 
This can be seen as follows. Fix $y\in \mathrm{int}(V_n)$, and for each $x\in \mathrm{int}(V_n)$ let $f(x)$ and $g(x)$ denote the left and right sides of \eqref{greencov}. Using \eqref{ffdens} it is easy to show that for any $x$,
\[
\ee(\phi_x|(\phi_z)_{z\ne x}) = \frac{1}{4}\sum_{z\in V_n, \; z\sim x} \phi_z.
\]
It follows that the function $f$ is discrete harmonic on $\mathrm{int}(V_n)\backslash\{y\}$, that is, 
\[
f(x) = \ee(\ee(\phi_x|(\phi_z)_{z\ne x}) \phi_y) = \frac{1}{4}\sum_{z\in V_n, \; z\sim x} f(z).
\]
Similarly, it is easy to show that $g(x)$ is discrete harmonic on $\mathrm{int}(V_n)\backslash\{y\}$ by conditioning on $\eta_1$. Again, putting
\[
\bar{\phi}_y = \ee(\phi_y|(\phi_z)_{z\ne y}) = \frac{1}{4}\sum_{z\in V_n, \; z\sim y}\phi_z,
\]
we have
\begin{align*}
f(y) - \frac{1}{4}\sum_{z\in V_n, \; z\sim y} f(z) &= \ee((\phi_y - \bar{\phi}_y)\phi_y)\\
&= \ee((\phi_y - \bar{\phi}_y)^2) + \ee((\phi_y - \bar{\phi}_y) \bar{\phi}_y) \\
&=  \ee(\var(\phi_y|(\phi_z)_{z\ne y})) + 0  \\
&= 1. 
\end{align*}
Similarly, conditioning on $\eta_1$, we can show 
\[
g(y) - \frac{1}{4}\sum_{z\in V_n, \; z\sim y} g(z) = 1.
\]
Thus, $f-g$ is discrete harmonic on $\mathrm{int}(V_n)$. But $f  = g = 0$ on $\partial V_n$. Thus we must have $f=g$ everywhere on $V_n$, which proves that the density \eqref{ffdens} indeed corresponds to the Gaussian field with covariance \eqref{greencov}.

It was shown by Bolthausen, Deuschel, and Giacomin (\cite{bdg01}, Lemma 1) that  as $n \ra \infty$, we have
\[
\max_{y\in V_n} \var(\phi_y) = \frac{2}{\pi}\log n + O(1).
\]
The unexpected and surprising fact, also in the same paper (\cite{bdg01}, Theorem~2), is that as $n \ra \infty$,
\[
\ee(\max_{y\in V_n}\phi_y) \sim 2\sqrt{\frac{2}{\pi}} \log n,
\]
exactly as if $\{\phi_y\}_{y\in V_n}$ were independent. Here $a_n \sim b_n$ means, as usual, that $\lim_{n\ra \infty} a_n/b_n = 1$. In our terminology, the DGFF with zero boundary condition is extremal.  Combined with Theorem \ref{extreme}, a direct consequence is the following result.
\begin{prop}
The discrete Gaussian free field on an $n \times n$ grid with zero boundary condition is superconcentrated (as $n\ra\infty$), meaning that
\[
\var(\max_{y\in V_n}\phi_y) = o(\log n).
\]
Consequently, the two dimensional DGFF is chaotic and has the multiple peaks property.
\end{prop}
\noindent  An immediate comment is that we cannot give a bound on the variance better than $o(\log n)$. This is because the result about the asymptotic behavior of $\ee(\max \phi_y)$ available from \cite{bdg01} does not include any explicit rate of convergence, which makes it impossible for us to get a more informative bound from Theorem \ref{extreme}. 

We next consider the free field on an $n\times n$ torus, where we can take advantage of the symmetry to prove that $\var(\max\phi(y)) = O(\log \log n)$.  

\subsection{DGFF on a torus}
Let $\mathbb{T}_n$ be the set $\{0, \ldots, n-1\}^2$  endowed with the graph structure of a torus, that it, $(a,b)$ and $(c,d)$ are adjacent if $a-c \equiv \pm 1 \pmod n$ and $b-d \equiv \pm 1\pmod n$. We wish to define a Gaussian free field on this graph. However, the graph has no natural boundary. The easiest (and perhaps the most natural) way to overcome this problem is to modify the definition~\eqref{greencov} of the covariance by replacing the stopping time $\tau_{\partial V_n}$ with a random time $\tau$ that is of the same order of magnitude as $\tau_{\partial V_n}$, but is independent of all else. Specifically, we prescribe 
\begin{equation}\label{green}
\cov(\phi_x, \phi_y) = \ee_x\biggl(\sum_{i=0}^\tau \ii_{\{\eta_i = y\}}\biggr),
\end{equation}
where $(\eta_i)_{i\ge 0}$ is a simple random walk on the torus started at $x$, and $\tau$ is a random variable independent of $(\eta_i)_{i\ge 0}$, which we take to be Geometrically distributed with mean $n^2$. The reason for the particular choice of the Geometric distribution is that it translates into a simple modification of the density \eqref{ffdens} by the introduction of a small `mass'; the new density function turns out to be 
\begin{equation}\label{ffdens2}
\exp\biggl( - \frac{(1-\frac{1}{n^2})}{8}\sum_{x\sim y} (\phi_x - \phi_y)^2 - \frac{1}{2n^2} \sum_x \phi_x^2\biggr).
\end{equation}
Here $x\sim y$ means that $x$ and $y$ are neighbors on the torus. To show that this density indeed corresponds to that of a centered Gaussian field with covariance \eqref{green}, we proceed as in the case of the DGFF with zero boundary condition. Fix $y\in \ttb_n$, and let $f(x)$ and $g(x)$ be the two sides of \eqref{green}. From \eqref{green} and  \eqref{ffdens2}, one can check using conditional expectations that for $x\ne y$,
\[
f(x) = \frac{n^2-1}{4n^2}\sum_{z\in V_n, \; z\sim x} f(z), \ \ \  g(x) = \frac{n^2-1}{4n^2}\sum_{z\in V_n, \; z\sim x} g(z).
\]
(The second identity holds because conditional on $\tau\ge1$, $\tau-1$ has the same distribution as $\tau$. This is where we use that $\tau$ has a Geometric law.) Again, using similar computations as before, it can be checked that
\begin{align*}
f(y) &= 1 + \frac{n^2-1}{4n^2}\sum_{z\in V_n, \; z\sim x} f(z), \\
g(y) &= 1 + \frac{n^2-1}{4n^2}\sum_{z\in V_n, \; z\sim y} g(z).
\end{align*}
Combining the last two displays, one can conclude that $|f-g|$ is a nonnegative strictly subharmonic function on $\ttb_n$, which implies that it must be zero everywhere.

Although we assign a small mass in our definition of the DGFF on the torus, we can still call it a `massless free field' in an asymptotic sense because the stopping times $\tau_{\partial V_n}$ and $\tau$ are both of order $n^2$ and it is not difficult to show that the covariances in the two models differ by $O(1)$.

Let us now state the main result of this section, which shows that the DGFF on the torus is superconcentrated, with an explicit bound of order $\log \log n$ on the variance of the maximum. 
\begin{thm}\label{torusvar}
Let $(\phi_x)_{x\in \ttb_n}$ be the DGFF on the torus defined above. For some universal constant $C$, we have
\[
\var(\max_{x\in \ttb_n} \phi_x) \le C\log \log n.
\]
\end{thm}
\noindent The proof of this result is via the use of hypercontractivity, more specifically Theorem \ref{hyper2}. The key advantage in the torus model is that we know that the location of the maximum is uniformly distributed. In the zero boundary situation, we had very little information about the location of the maximum.

Let us now proceed to prove Theorem \ref{torusvar}. The first step is a basic observation about the simple symmetric random walk on $\zz$.
\begin{lmm}\label{binom}
Suppose $(\alpha_i)_{i\ge 0}$ is a simple symmetric random walk on $\zz$, starting at $0$. Then for any $k\in \zz$ and $i\ge 1$, we have
\[
\pp(\alpha_i = k) \le \frac{Ce^{-k^2/4i}}{\sqrt{i}}
\]
where $C$ is a universal constant. 
\end{lmm}
\begin{proof}
If $|k|> i$ or $k \not \equiv i\pmod 2$, then $\pp(\alpha_i = k) = 0$ and we have nothing to prove. If $|k|=i$, we have $\pp(\alpha_i = k) = 2^{-i}$, which is consistent with the statement of the lemma. In all other cases, 
\[
\pp(\alpha_i = k) = {i \choose \frac{i+k}{2}} 2^{-i}. 
\]
Using the Stirling approximation, we get
\begin{align*}
\pp(\alpha_i = k) &\le \frac{C}{\sqrt{i}}\exp\biggl(-\frac{i+k+1}{2}\log \biggl(1+\frac{k}{i}\biggr) - \frac{i-k+1}{2}\log \biggl(1- \frac{k}{i}\biggr)\biggr)\\
&= \frac{C}{\sqrt{i}} \exp\biggl( - i I\biggl(\frac{k}{i}\biggr) - \frac{1}{2}\log \biggl(1-\frac{k^2}{i^2}\biggr)\biggr), 
\end{align*}
where
\[
I(x) = \frac{1+x}{2}\log (1+x) + \frac{1-x}{2}\log (1-x). 
\]
The function $I$ has a power series expansion 
\[
I(x) = \sum_{p=1}^\infty\frac{x^{2p}}{(2p-1)2p}, \ \ x\in (-1,1). 
\]
In particular, $I(x) \ge x^2/2$. This inequality suffices to prove the lemma when, say, $k \le i/2$. On the other hand, if $i/2 < k < i$ and $i$ is so large that $\log i \le i/8$ (which can be assumed by choosing a suitably large $C$), we have 
\[
-\frac{1}{2}\log \biggl(1-\frac{k^2}{i^2}\biggr) \le \frac{\log i}{2}\le \frac{i}{16} \le \frac{k^2}{4i}, 
\]
which implies that 
\[
\pp(\alpha_i = k) \le \frac{2}{\sqrt{2\pi i}} \exp\biggl(- \frac{k^2}{2i} + \frac{k^2}{4i}\biggr).
\]
This completes the proof. 
\end{proof}
We are going to use random walks on $\zz^2$ to produce random walks on~$\ttb_n$. For this purpose, we observe that there is a natural map $Q_n: \zz^2 \ra \ttb_n$ which takes a point $(x_1,x_2)\in \zz^2$ to the unique point $(x_1', x_2')$ in $\ttb_n$ satisfying $x_1\equiv x_1' \pmod n$ and $x_2 \equiv x_2'\pmod n$. For any $x,y\in\ttb_n$, define the toric Euclidean distance $d_n(x,y)$ as:
\[
d_n(x,y) := d(x, Q_n^{-1}(y)),
\]
where $d(x,A)$ is the usual Euclidean distance of a point $x$ from a set $A$. It is not difficult to verify that actually $d_n(x,y) = d(Q_n^{-1}(x), Q_n^{-1}(y))$ and therefore the definition is symmetric in $x$ and $y$.
\begin{lmm}\label{return}
Let $(\eta_i)_{i\ge 0}$ be a simple symmetric random walk on $\ttb_n$. Then for any $x, y\in \ttb_n$ and any $i\ge 1$, we have
\[
\pp_x(\eta_i = y) \le
\begin{cases} 
Ci^{-1}e^{-d_n(x,y)^2/4i} & \text{ if } i\le n^2,\\
Cn^{-2} &\text{ if } i> n^2,
\end{cases}
\] 
where $C$ is a universal constant and $\pp_x$ denotes the law of the random walk starting from $x$. 
\end{lmm}
\begin{proof}
Let $(\beta_i)_{i\ge 0}$ be a simple symmetric random walk on $\zz^2$. Then a random walk on the torus is easily obtained as $\eta_i = Q_n(\beta_i)$. For any $x, y\in \ttb_n$, we have
\begin{equation}\label{tor1}
\pp_x(\eta_i = y) = \sum_{z\in Q_n^{-1}(y)} \pp_x(\beta_i = z). 
\end{equation}
Now, for any $x$ and $z$, Lemma \ref{binom} shows that
\begin{equation}\label{tor2}
\pp_x(\beta_i = z)\le \frac{C}{i}\exp\biggl(-\frac{d(x,z)^2}{4i}\biggr),
\end{equation}
where $d(x,z)$ is the Euclidean distance between $x$ and $z$. Now fix $x,y\in \ttb_n$, and let $z$ be the nearest point to $x$ in $Q_n^{-1}(y)$. Then, if $x = (x_1,x_2)$ and $z = (z_1,z_2)$, we have by \eqref{tor1} and \eqref{tor2} that 
\[
\pp_x(\eta_i = y) \le \frac{C}{i}\sum_{k_1, k_2 \in \zz}\exp\biggl(-\frac{(x_1 - z_1 + k_1 n)^2 + (x_2 - z_2 + k_2 n)^2}{4i}\biggr). 
\]
It is easy to see that $|x_j - z_j|\le n/2$ for $j = 1,2$ (otherwise, we can choose a `better' $z_j$ so that $z$ is closer to $x$.) Thus, for $j=1,2$,
\begin{align*}
(x_j-z_j + k_j n)^2 &= (x_j - z_j)^2 + k_j^2 n^2 + 2(x_j - z_j)k_j n\\
&\ge (x_j - z_j)^2 + |k_j|(|k_j| - 1)n^2. 
\end{align*}
Therefore,
\begin{align*}
\pp_x(\eta_i = y) &\le \frac{C}{i}e^{-d(x,z)^2/4i} \sum_{k_1, k_2 = 0}^\infty \exp\biggl(-\frac{(k_1(k_1-1) + k_2(k_2-1))n^2}{4i}\biggr)\\
&\le \frac{C}{i}e^{-d(x,z)^2/4i}\biggl(1 + \sum_{k=1}^\infty e^{-k^2n^2/4i} + \sum_{r,s = 1}^\infty e^{-(r^2 + s^2) n^2/4i}\biggr).
\end{align*}
Comparing the last two terms with integrals, we get
\begin{align*}
\pp_x(\eta_i = y) &\le \frac{C}{i}e^{-d(x,z)^2/4i}\biggl(1 + \int_\rr e^{-u^2 n^2/4i} du + \int_{\rr^2} e^{-(u^2 + v^2)n^2/4i} \;du \;dv\biggr)\\
&\le \frac{C}{i}e^{-d(x,z)^2/4i}\biggl(1  + \frac{\sqrt{i}}{n} + \frac{i}{n^2}\biggr). 
\end{align*}
It is not difficult to verify by considering the cases $i \le n^2$ and $i > n^2$ that this completes the proof.
\end{proof}
\begin{lmm}\label{toruscov}
For any $x\ne y\in \ttb_n$, we have
\[
0\le \cov(\phi_x, \phi_y) \le C\log \frac{n}{d_n(x,y)} + C.
\]
and $\var(\phi_x)\le C\log n$, where $C$ is a universal constant. 
\end{lmm}
\begin{proof}
From the representation \eqref{green}, we see that the covariances are nonnegative. Now fix two distinct points $x, y\in \ttb_n$, and let $d = d_n(x,y)$. From~\eqref{green}, we have
\begin{align*}
\cov(\phi_x, \phi_y) &= \sum_{i=0}^\infty \pp_x(\eta_i=y)\pp(\tau\ge i)\\
&\le \sum_{i=1}^{n^2}\frac{Ce^{-d^2/4i}}{i} + \sum_{i=n^2}^\infty \frac{C}{n^2}\biggl(1-\frac{1}{n^2}\biggr)^i.
\end{align*}
Clearly, the second sum can be bounded by a constant that does not depend on $n$. For the first, note that by the inequality $e^{-x} \le x^{-1}$ that holds for $x\ge 1$, we have 
\begin{align*}
\sum_{1\le i\le n^2} \frac{Ce^{-d^2/4i}}{i} &=  \sum_{1\le i \le d^2} \frac{Ce^{-d^2/4i}}{i} + \sum_{d^2 < i \le n^2} \frac{Ce^{-d^2/4i}}{i}\\
&\le \sum_{1\le i \le d^2} \frac{C}{d^2} + \sum_{d^2 < i \le n^2} \frac{C}{i}\\
&\le C + C (\log n^2 - \log d^2).
\end{align*}
The bound on the variance follows similarly. This completes the proof. 
\end{proof}
\begin{proof}[Proof of Theorem \ref{torusvar}]
Fix some $r \le C\log n$, where $C$ is the same constant as in Lemma \ref{toruscov}. If $\cov(\phi_x,\phi_y) \ge r$, then by Lemma \ref{toruscov},
\begin{equation}\label{sbd}
d_n(x,y) \le ne^{1-(r/C)} =: s.
\end{equation}
Let $A$ be an $s$-net of $\ttb_n$ for the metric $d_n$ (i.e.\ a set of points that are mutually separated from each other by distance $> s$, and is maximal with respect to this property). Let $\mathcal{C}(r)$ be the collection of all $2s$-balls around the points of $A$. Then by \eqref{sbd} we see that whenever $\cov(\phi_x,\phi_y)\ge r$, we must have that $x,y\in D$ for some $D\in \mathcal{C}(r)$. By symmetry, we see that the probability of the maximum being at any point $z\in \ttb_n$ is exactly $n^{-2}$. Therefore, the probability of the maximum being in any given $D\in \mathcal{C}(r)$ is bounded by $Ks^2/n^2$, where $K$ is a universal constant. Thus, in the terminology of Theorem \ref{hyper2}, 
\[
\rho(r)\le \frac{Ks^2}{n^2} = K'e^{-2r/C},
\]
where $K' = K e^2$. We can assume that $K' > 1$. 
Again, if $x\in \ttb_n$ and $D\in \mathcal{C}(r)$ contains $x$, then the center of $D$ must be at distance $\le 2s$ from $x$. Now, the centers of the members of $\mathcal{C}(r)$ are separated by distance $\ge s$ from each other. Clearly, the maximum number of $s$-separated points in a $2s$-ball can be bounded by a universal constant that does not depend on $n$ and $s$. It follows that the number of members of $D$ that contain $x$ is bounded by a universal constant. Therefore, in the notation of Theorem \ref{hyper}, $\mu(r)$ is bounded by a universal constant. Using the bounds on $\rho(r)$ and $\mu(r)$ in Theorem \ref{hyper2} and the inequality $1-x \le -\log x$ for $x> 0$, we see that for some universal constant $\kappa$, 
\[
\var(\max_x\phi_x)\le C\log K' + \int_{C \log K'}^{C\log n} \frac{\kappa}{(2r/C) - \log K'}\; dr.
\]
It is easy to see that the right hand side is bounded by a constant multiple of $\log \log n$. This completes the proof. 
\end{proof}

\section{Example: Gaussian fields on Euclidean spaces}\label{euclidean}
Consider a stationary centered Gaussian process $(X_n)_{n\ge 0}$. If we have $\cov(X_0, X_n)$ decaying to zero faster than $1/\log n$ as $n \ra \infty$, then it is known at least since the work of Berman \cite{berman64} that $M_n := \max\{X_0,\ldots,X_n\}$ has fluctuations of order $(\log n)^{-1/2}$ and upon proper centering and scaling, converges to the Gumbel distribution in the limit. This result has seen considerable generalizations for one dimensional Gaussian processes, both in discrete  and continuous time. Some examples are \cite{pickands67}, \cite{pickands69}, \cite{pickands69a}, and \cite{mittalylvisaker75}. For a survey of the classical literature, we refer to the book \cite{llr83}. 

The question is considerably harder in dimensions higher than one. A~large number of sophisticated results and techniques for analyzing the behavior of the maxima of higher dimensional {\it smooth} Gaussian fields are now known; see the excellent recent book of Adler and Taylor \cite{adlertaylor07} for a survey. Here `smooth' usually means twice continuously differentiable.  However, in the absence of smoothness, maxima of high dimensional Gaussian fields are still quite intractable. If only the expected size of the maximum is of interest, advanced techniques exist (see the book on chaining by Talagrand~\cite{talagrand05}). The question of fluctuations is much more difficult. In fact, for general (nonsmooth) processes, even the one-dimensional work of Pickands \cite{pickands67, pickands69, pickands69a} is considerably nontrivial. 

One basic question one may ask is the following: what is a sufficient condition for the fluctuation of the maximum in a box of side length $T$ to decrease like $(\log T)^{-1/2}$? In other words, when does the maximum behave as if it were the maximum of a collection of i.i.d.\ Gaussians, one per each unit area in the box? Classical theory \cite{mittalylvisaker75} tells  that this is true for stationary one-dimensional Gaussian processes whenever the correlation between $X_0$ and $X_T$ decreases at least as fast as $1/\log T$. Note that this requirement for the rate of decay is rather mild, considering that it ensures that the maximum behaves just like the maximum of independent variables. 

In the dimensions $\ge 2$, the above question is unresolved. Here questions may also arise about maxima over subsets that are not necessarily boxes. Moreover, what if the correlation decays slower than $1/\log T$? In this section, we attempt to answer these questions. Our achievements are modest: we only have upper bounds on the variances. The issue of limiting distributions is not solved here.  

Let $\bbx = (X_\bu)_{\bu\in \rr^d}$ be a centered Gaussian field on $\rr^d$ with $\ee(X_\bu^2)=1$ for each $\bu$. For any Borel set $A \subseteq \rr^d$, let
\[
M(A) := \sup_{\bu\in A} X_\bu, \ \ m(A) := \ee(M(A)).
\]
For any $\bu\in \rr^d$ and $r> 0$, let $B(\bu,r)$ denote the open ball of radius $r$ and center $\bu$. Assume that
\begin{equation}\label{as1}
L := \sup_{\bu\in\rr^d} m(B(\bu, 1)) < \infty.
\end{equation}
Note that in particular, the above condition is satisfied when the field is stationary and continuous. Next, suppose $\phi:[0,\infty)\ra \rr$ is a decreasing function such that for all $\bu,\mathbf{v}\in \rr^d$,
\[
\cov(X_\bu, X_\mathbf{v}) \le \phi(|\bu-\mathbf{v}|),
\]
where $|\bu-\mathbf{v}|$ denotes the Euclidean distance between $\bu$ and $\mathbf{v}$. Assume that  
\begin{equation}\label{as2}
\lim_{s\ra \infty} \phi(s)=0. 
\end{equation}
For a set $A\subseteq \rr^d$ and $\epsilon > 0$, let $N(A, \epsilon)$ be the maximum number of points that can be found in $A$ such that any two them are separated by distance greater than $\epsilon$ (such a collection is usually called an $\epsilon$-net of $A$). When $\epsilon =1$, we simply write $N(A)$ instead of $N(A,1)$. The following theorem is the main result of this section.
\begin{thm}\label{fieldmax}
Assume \eqref{as1} and \eqref{as2}. Then for any Borel set $A\subseteq \rr^d$ such that $\mathrm{diam}(A) > 1$, we have
\[
\var(M(A))\le C_1(\phi,d) \biggl( \phi(N(A)^{C_2(\phi,d)}) + \frac{1}{\log N(A)}\biggr),
\]
where $C_1(\phi,d)$ and $C_2(\phi,d)$ are constants that depend only on the function $\phi$ and the dimension $d$ (and not on the set $A$), and $N(A)$ is defined above. 
\end{thm}
\noindent {\it Remarks.} The first observation is that if $\phi(s)$ decreases at least as fast as $1/\log s$, then the above result clearly shows that
\[
\var(M(A))\le \frac{C(\phi,d)}{\log N(A)},
\]
where $C(\phi,d)$ is some constant that depend only on $\phi$ and $d$. In particular, it gives a broad generalization of the classical results about fluctuations in one dimension \cite{berman64, pickands67, pickands69, pickands69a, mittalylvisaker75}. 
An additional observation is that the first term in the bound can dominate the second only if $\phi(s)$ decreases slower than $1/\log s$. 

Before we embark on the proof of Theorem \ref{fieldmax}, we need some simple upper and lower bounds on the expected value of $M(A)$. 
\begin{lmm}\label{twobd}
Under \eqref{as1} and \eqref{as2}, for any Borel set $A\subseteq \rr^d$ such that $N(A) \ge 2$, we have
\[
c_1(\phi,d)\sqrt{\log N(A)} \le m(A) \le c_2(\phi,d)\sqrt{\log N(A)},
\]
where $c_1(\phi,d)$ and $c_2(\phi,d)$ are positive constants that depend only on the function~$\phi$ and the dimension $d$. 
\end{lmm}
\begin{proof}
The upper bound follows easily  from  a combination of  assumption~\eqref{as1}, the tail bound from Proposition \ref{borell} for the maximum of the field in unit balls, and an argument similar to the proof of Lemma \ref{max}. 

For the lower bound, first let $s > 1$ be a number such that $\phi(s) < 1/2$. Such an $s$ exists by the assumption that $\lim_{s\ra \infty} \phi(s) = 0$. Next, let $B$ be a $1$-net of $A$ and $D$ be an $s$-net of $B$. For each $\bx \in D$, the $s$-ball around $\bx$ can contain at most $k$ points of $B$, where $k$ is a fixed number that depends only on $s$ and the dimension $d$. Thus,
\[
|D| \ge \frac{|B|}{k} = \frac{N(A)}{k}. 
\]
Since $\phi(s) < 1/2$ and $\ee(X_\bu^2)=1$ for each $\bu$, by the Sudakov minoration technique from Section \ref{basicfacts} we have
\[
m(A) \ge m(D) \ge C \sqrt{\log |D|} \ge c_1(\phi,d) \sqrt{\log N(A)},
\]
where $D$ is a universal constant and $c_1(\phi,d)$ is a constant depending only on the function $\phi$ and the dimension $d$.
\end{proof}
\begin{proof}[Proof of Theorem \ref{fieldmax}]
Let $c_1$ and $c_2$ be the constants from Lemma \ref{twobd}. Put
\[
s := N(A)^{\frac{1}{8}(c_1/c_2)^2},
\]
and assume that $N(A)$ is so large that $s > 2$. Let $r = \phi(s)$. Take any maximal $s$-net of $A$, and let $\mc(r)$ be the set of $2s$-balls around the points in the net. It is easy to verify using the definition of $s$ and the decreasing nature of $\phi$, that $\mc(r)$ is a covering of $A$ satisfying the conditions of Theorem~\ref{hyper2}. Now take any $D\in \mc(r)$. Since $D$ is a $2s$-ball and $s > 2$, by Lemma \ref{twobd} we have
\[
m(D) \le c_2\sqrt{\log 2s}\le c_2\sqrt{2\log s} = \frac{c_1}{2}\sqrt{\log N(A)}.
\]
Also, by Lemma \ref{twobd}, we have $m(A) \ge c_1\sqrt{\log N(A)}$. Thus, using the notation of Theorem \ref{hyper2}, we have by Proposition \ref{borell} that
\begin{align*}
p(D) &\le \pp\biggl(M(D) \ge m(D) + \frac{m(A)-m(D)}{2}\biggr) \\
&\qquad + \pp\biggl(M(A) \le m(A) - \frac{m(A)-m(D)}{2}\biggr)\\
&\le 2\exp\biggl(-\frac{(m(A)-m(D))^2}{8}\biggr)\\
&\le 2\exp\biggl(-\frac{c_1^2\log N(A)}{32}\biggr).
\end{align*}
Since the above bound does not depend on $D$, it serves as a bound on $\rho(r)$. Let us now get a bound on $\mu(r)$. Take any $\bu \in A$. Then the center of any $D\in \mc(r)$ that contains $\bu$ is a point in $B(\bu, 2s)$. The centers are mutually separated by distance more than $s$. Hence, the number of $D\in \mc(r)$ that can contain $\bu$ is bounded by $N(B(\bu,2s), s)$, which, by scaling symmetry, is equal to $N(B(\mathbf{0}, 2), 1)$. Since $\mu(r)$ is the expected number of elements of $\mc(r)$ that contain the maximizer of $\bbx$ in $A$, therefore we have that $\mu(r) \le c_3$, where $c_3$ is a constant that depends only on the dimension $d$. Combining the bounds, we see that whenever $N(A) \ge c_4$, we have
\[
\frac{\mu(r)}{|\log \rho(r)|} \le \frac{c_5}{\log N(A)},
\]
where $c_4$ and $c_5$ are constants depending only on $\phi$ and $d$. Note that we have this bound only for one specific value of $r$ defined above. Now, if $r' > r$, and we define $\mc(r')$ the same way as we defined $\mc(r)$, then clearly $\mc(r')$ would also be a cover of $A$ satisfying the requirements of Theorem \ref{hyper2}, and we would have $\rho(r')\le \rho(r)$. Noting this, and the fact that $|(1-x)/\log x|\le 1$ for all $x\in (0,1)$, we have by Theorem \ref{hyper2} that
\[
\var(M(A))\le c_6 \biggl(r + \frac{1}{\log N(A)}\biggr),
\]
for some constant $c_6$ depending only on $\phi$ and $d$. Of course this holds only if $N(A)\ge c_4$, but this condition can now be dropped by increasing the value of $c_6$, since we always have $\var(M(A))\le 1$. Plugging in the value of $r$, the proof is done.
\end{proof}

\section{Open problems}
In some sense, this paper raises more questions than it solves. Some of these problems are listed below. At present, the author does not know how to solve any of these.
\begin{enumerate}
\item Prove chaos in the original Sherrington-Kirkpatrick model, where $\xi(x) = x^2$. This is perhaps the most important open question that may be solvable by suitable extensions of the methods of this paper. 
\item Prove chaos in directed polymers in dimensions $\ge 3$. Certain parts of the current proof do not work in higher than $2$ dimensions, but other parts are okay. Since the Benjamini-Kalai-Schramm approach is inherently dimension-free, this seems to be a more doable open problem. 
\item Improve the fluctuation exponent in directed polymers. The current proof via hypercontractivity gives only a $\log n$ correction. Although this suffices to prove chaos, it is not satisfactory.
\item Improve the variance bound in the discrete Gaussian free field with zero boundary condition. Although the case of the torus is interesting, the zero boundary condition is the more important one. Also, even in the torus, is $\log \log n$ the correct order? The author thinks $O(1)$ is more likely to be the right answer. 
\item Show that chaos implies multiple peaks even when the correlations are not all nonnegative. This involves getting a bound on the second moment of $R(I^0,I^t)$, which we do not know how to derive at present.
\item Find a suitable multiple peaks condition that is actually equivalent to chaos and superconcentration. There is a possibility that there does not exist any such condition.
\item Find a more intuitive and less analytical (rigorous) proof of the equivalence of superconcentration and chaos. 
\item A very interesting question in the context of multiple peaks is the issue of `bridges': do the multiple peaks exist as disconnected islands, or do there exist `bridges' that allow one to move from one peak to another without `climbing down'? This question is particularly relevant for the Kauffman-Levin fitness model, for obvious evolutionary implications. Our method of proving the existence of multiple peaks suggests an obvious way to prove the existence of bridges, but the author does not yet know how to carry out the program for the $NK$ model.
\end{enumerate}

\vskip.2in
\noindent{\bf Acknowledgments.} The author thanks Persi Diaconis, Michel Talagrand, Daniel Fisher,   Michel Ledoux, David Aldous, Itai Benjamini, Steve Evans,  
Elchanan Mossel, Boris Tsirelson,  Oded Schramm, Alexei Borodin, Ivan Nourdin, 
Jean-Philippe Bouchaud, Rongfeng Sun, and Giorgio Parisi 
for a number of 
useful discussions and communications. Special thanks go to Michel Ledoux for a major simplification of the proof of Theorem \ref{varconv}, to Rongfeng Sun for pointing out a serious mistake in the proof of Theorem \ref{poly} in the first draft (corrected by the author in the current manuscript), and to Itai Benjamini and Michel Talagrand for suggesting that the sensitivity to noise in glassy systems is a problem worth investigating.


\begin{thebibliography}{99}

\bibitem{adlertaylor07} {\sc Adler, R.~J.} and {\sc Taylor, J.~E. (2007).} {\it Random fields and geometry.} Springer Monographs in Mathematics. {\it Springer, New York.}

\bibitem{alr87} {\sc Aizenman, M., Lebowitz, J.~L.,} and {\sc Ruelle, D. (1987).} Some rigorous results on the Sherrington-Kirkpatrick spin glass model. {\it Comm. Math. Phys.} {\bf 112} no. 1, 3--20. 



\bibitem{albeveriozhou96} {\sc Albeverio, S.} and {\sc  Zhou, X.~Y. (1996).} A martingale approach to directed polymers in a random environment. {\it J. Theoret. Probab.} {\bf 9} no. 1, 171--189. 







\bibitem{aneetal00} {\sc An\'e, C., Blach\`ere, S.,  Chafa\"i, D., Foug\`eres, P., Gentil, I., Malrieu, F., Roberto, C.} and {\sc Scheffer, G. (2000).}  Sur les in\'egalit\'es de Sobolev logarithmiques. Panoramas et Synth\`eses, {\bf 10}. {\it Soci\'et\'e Math\'ematique de France, Paris}.

\bibitem{aubrun05} {\sc Aubrun, G. (2005).} A sharp small deviation inequality for the largest eigenvalue of a random matrix. {\it S\'eminaire de Probabilit\'es XXXVIII,} 320--337, Lecture Notes in Math., {\bf 1857}, {\it Springer, Berlin}.



\bibitem{bks01} {\sc Benjamini, I., Kalai, G.,} and {\sc  Schramm, O. (2001).} Noise sensitivity of Boolean functions and applications to percolation. {\it Inst. Hautes \'Etudes Sci. Publ. Math. No. 90 (1999),} 5--43.


\bibitem{bks03} {\sc Benjamini, I., Kalai, G.,} and {\sc Schramm, O. (2003).} First passage percolation has sublinear distance variance. {\it Ann. Probab.} {\bf 31} no. 4, 1970--1978.

\bibitem{berman64} {\sc Berman, S.~M. (1964).} Limit Theorems for the Maximum Term in Stationary Sequences. {\it Ann. Math. Statist.,} {\bf 35} no. 2, 502--516. 


\bibitem{biggins77} {\sc Biggins, J.~D. (1977).} Chernoff's theorem in the branching random walk. {\it J. Appl. Probab.} {\bf 14} 630--636. 


\bibitem{bolthausen89} {\sc Bolthausen, E. (1989).} A note on the diffusion of directed polymers in a random environment. {\it Comm. Math. Phys.} {\bf 123} no. 4, 529--534.

\bibitem{bdg01} {\sc Bolthausen, E., Deuschel, J.-D.,} and {\sc Giacomin, G. (2001).} Entropic repulsion and the maximum of the two-dimensional harmonic crystal. {\it Ann. Probab.} {\bf 29} no. 4, 1670--1692. 


\bibitem{borell75} {\sc Borell, C. (1975).} The Brunn-Minkowski inequality in Gauss space. {\it Invent. Math.,} {\bf 30} 205--216. 


\bibitem{braymoore87} {\sc Bray, A.~J.} and {\sc Moore, M.~A. (1987).} Chaotic Nature of the Spin-Glass Phase. {\it Phys. Rev. Lett.,} {\bf 58} no. 1, 57--60. 


\bibitem{carmonahu02} {\sc Carmona, P.} and {\sc Hu, Y. (2002).} On the partition function of a directed polymer in a Gaussian random environment. {\it Probab. Theory Related Fields} {\bf 124} no. 3, 431--457. 


\bibitem{carmonahu04} {\sc Carmona, P.} and {\sc Hu, Y. (2004).} Fluctuation exponents and large deviations for directed polymers in a random environment. {\it Stochastic Process. Appl.} {\bf 112} no. 2, 285--308. 

\bibitem{catorgroeneboom06} {\sc Cator, E.} and {\sc  Groeneboom, P. (2006).} Second class particles and cube root asymptotics for Hammersley's process. {\it Ann. Probab.} {\bf 34} no. 4, 1273--1295. 


\bibitem{chen82} {\sc Chen, L.~H.~Y. (1982).} An inequality for the multivariate normal distribution. {\it J. Multivariate Anal.} {\bf 12} 306--315. 

\bibitem{chernoff81}
{\sc Chernoff, H. (1981).} A note on an inequality involving the normal distribution. {\it Ann. Probab.} {\bf 9} 533--535. 
 

\bibitem{cometsneveu95} {\sc Comets, F.} and {\sc Neveu, J. (1995).} The Sherrington-Kirkpatrick model of spin glasses and stochastic calculus: the high temperature case. {\it Comm. Math. Phys.} {\bf 166} no. 3, 549--564.


\bibitem{cometsshigayoshida04} {\sc Comets, F., Shiga, T.,} and {\sc Yoshida, N. (2004).}  Probabilistic analysis of directed polymers in a random environment: a review. {\it Stochastic analysis on large scale interacting systems,} 115--142, {\it Adv. Stud. Pure Math.,} {\bf 39} Math. Soc. Japan, Tokyo.

\bibitem{cometsyoshida06} {\sc Comets, F.} and {\sc  Yoshida, N. (2006).} Directed polymers in random environment are diffusive at weak disorder. {\it Ann. Probab.} {\bf 34} no. 5, 1746--1770.



\bibitem{crisantietal92} {\sc Crisanti, A., Paladin, G., Sommers, H.-J.} and {\sc Vulpiani, A. (1992).}  Replica trick and fluctuations in disordered systems. {\it J. Phys. I France} {\bf 2} 1325--1332.

\bibitem{durrettlimic03} {\sc Durrett, R.} and {\sc Limic, V. (2003).} Rigorous results for the $NK$ model. {\it Ann. Probab.} {\bf 31} no. 4, 1713--1753. 

\bibitem{evanssteinsaltz02} {\sc Evans, S.~N.} and {\sc Steinsaltz, D. (2002).} Estimating some features of $NK$ fitness landscapes. {\it Ann. Appl. Probab.} {\bf 12} no. 4, 1299--1321.

\bibitem{fisherhuse91} {\sc Fisher, D.~S.} and {\sc Huse, D.~A. (1991).} Directed paths in a random potential. {\it Phys. Rev. B,}  {\bf 43} no. 13, 10728--10742.

\bibitem{frohlichzegarlinski87} {\sc Fr\"ohlich, J.} and {\sc Zegarli\'nski, B. (1987).} Some comments on the Sherrington-Kirkpatrick model of spin glasses. {\it Comm. Math. Phys.} {\bf 112} no. 4, 553--566.

\bibitem{gardner85} {\sc Gardner, E. (1985).} Spin glasses with $p$-spin interactions. {\it Nuclear Phys. B} {\bf 257} no. 6, 747--765.

\bibitem{giacomin00} {\sc Giacomin, G. (2000).}  Anharmonic lattices, random walks and random interfaces. Recent Research 
Developments in Statistical Physics, Transworld Research, {\bf I}, 97--118. 

\bibitem{gordon41} {\sc Gordon, R.~D. (1941).} Values of Mills' ratio of area to bounding ordinate and of the normal probability integral for large values of the argument. {\it Ann. Math. Statist.} {\bf 12} 364--366. 

\bibitem{gravnertracywidom01} {\sc Gravner, J.,  Tracy, C.~A.,} and {\sc Widom, H. (2001).} Limit theorems for height fluctuations in a class of discrete space and time growth models. {\it J. Statist. Phys.} {\bf 102} no. 5-6, 1085--1132.

\bibitem{gross75} {\sc Gross, L. (1975).} Logarithmic Sobolev inequalities. {\it Amer. J. Math.} {\bf 97} no. 4, 1061--1083.

\bibitem{grossmezard84} {\sc Gross, D.} and {\sc M\'ezard, M. (1984).} The simplest spin glass. {\it Nuclear Phys. B} {\bf 240} 431--452. 

\bibitem{guerra02} {\sc Guerra, F.} and {\sc Toninelli, F.~L. (2002).} The thermodynamic limit in mean field spin glass models. {\it Comm. Math. Phys.} {\bf 230} 71--79. 

\bibitem{guerra03} {\sc Guerra, F. (2003).} Broken replica symmetry bounds in the mean field spin glass model. {\it Comm. Math. Phys.} {\bf 233} 1--12. 

\bibitem{guionnetzegarlinski03} {\sc Guionnet, A.} and {\sc Zegarlinski, B. (2003).} Lectures on logarithmic Sobolev inequalities. {\it S\'eminaire de Probabilit\'es, XXXVI,} 1--134, Lecture Notes in Math., {\bf 1801}, {\it Springer, Berlin}. 

\bibitem{houdre95} {\sc Houdr\'e, C. (1995).} Some applications of covariance identities and inequalities to functions of multivariate normal variables. {\it J. Amer. Statist. Assoc.} {\bf 90} no. 431, 965--968. 

\bibitem{houdrekagan95} {\sc Houdr\'e, C.} and {\sc  Kagan, A. (1995).} Variance inequalities for functions of Gaussian variables. {\it J. Theoret. Probab.} {\bf 8} 23--30. 

\bibitem{husehenleyfisher85} {\sc Huse, D.~A., Henley, C.~L.} and {\sc Fisher, D.~S. (1985).} Huse, Henley and Fisher respond. {\it Phys. Rev. Lett.} {\bf 55} no. 26, 2924.


\bibitem{hwafisher94} {\sc Hwa, T.} and {\sc Fisher, D.~S. (1994).} Anomalous fluctuations of directed polymers in random media. {\it Phys. Rev. B,} {\bf 49} no. 5, 3136--3155. 


\bibitem{imbriespencer88} {\sc  Imbrie, J.~Z.} and {\sc Spencer, T. (1988).} Diffusion of directed polymers in a random environment. {\it J. Statist. Phys.} {\bf 52} no. 3--4, 609--626.


\bibitem{johansson00} {\sc Johansson, K. (2000).}  Shape fluctuations and random matrices. {\it Comm. Math. Phys.} {\bf 209} no. 2, 437--476.


\bibitem{johansson01} {\sc Johansson, K. (2001).}  Discrete orthogonal polynomial ensembles and the Plancherel measure. {\it Ann. Math. (2),} {\bf 153}  no. 1, 259--296.

\bibitem{kauffmanlevin87} {\sc Kauffman, S.~A.} and {\sc Levin, S.~A. (1987).} Towards a general theory of adaptive walks on rugged landscapes. {\it J. Theoret. Biol.} {\bf 128} 11--45. 


\bibitem{kifer97} {\sc Kifer, Y. (1997).} The Burgers equation with a random force and a general model for directed polymers in random environments. {\it Probab. Theory Related Fields} {\bf  108} no. 1, 29--65. 




\bibitem{kondor83} {\sc Kondor, I. (1983).} Parisi's mean-field solution for spin glasses as an analytic 
continuation in the replica number.  {\it J. Phys. A,} {\bf 16} L127--L131.

\bibitem{llr83} {\sc Leadbetter, M.~R., Lindgren, G.} and {\sc Rootz\'en, H. (1983).} {\it Extremes and related properties of random sequences and processes.} Springer Series in Statistics. {\it Springer-Verlag, New York - Berlin.} 

\bibitem{ledoux01} {\sc Ledoux, M. (2001).} {\it The Concentration of Measure Phenomenon. } Amer. Math. Soc., Providence, RI.


\bibitem{ledoux03} {\sc Ledoux, M. (2003).} A remark on hypercontractivity and tail inequalities for the largest eigenvalues of random matrices. {\it S\'eminaire de Probabilit\'es XXXVII,} 360--369, Lecture Notes in Math., {\bf 1832},{\it Springer, Berlin}.

\bibitem{limicpemantle04} {\sc Limic, V.} and {\sc Pemantle, R. (2004).}
More rigorous results on the Kauffman-Levin model of evolution. {\it Ann. Probab.} {\bf 32} no. 3A, 2149--2178. 

\bibitem{mckayetal82} {\sc McKay, S.~R., Berger, A.~N.,} and {\sc Kirkpatrick, S. (1982).} Spin-Glass Behavior in Frustrated Ising Models with Chaotic Renormalization-Group Trajectories. {\it Phys. Rev. Lett.} {\bf 48}, 767--770. 


\bibitem{mehta91} {\sc Mehta, M. L.} (1991). {\it Random Matrices.} 2nd ed. Academic Press, Boston, MA.

\bibitem{mezard90} {\sc M\'ezard, M. (1990).}  
On the glassy nature of random directed polymers 
in two dimensions. {\it J. Phys. France} {\bf 51} 1831--1846.

\bibitem{mezardetal87} {\sc M\'ezard, M., Parisi, G.,} and {\sc Virasoro, M.~A. (1987).} {\it Spin glass theory and beyond.} World Scientific Lecture Notes in Physics, {\bf 9}. {\it World Scientific Publishing Co., Inc., Teaneck, NJ.}

\bibitem{mittalylvisaker75} {\sc Mittal, Y.} and  {\sc Ylvisaker, D. (1975).} Limit distributions for the maxima of stationary Gaussian processes. {\it Stoch. Proc. Appl.} {\bf 3} 1--18.


\bibitem{nash58} {\sc Nash, J. (1958).} Continuity of solutions of parabolic and elliptic equations. {\it Amer. J. Math.} {\bf 80} no. 4, 931--954. 


\bibitem{nelson66} {\sc Nelson, E. (1966).} A quartic interaction in two dimensions. {\it Mathematical 
Theory of Elementary Particles (Proc. Conf., Dedham, Mass., 1965),} M.I.T. 
Press, Cambridge, Mass.,  69--73. 
\bibitem{nelson73} {\sc Nelson, E. (1973).} The free Markoff field. {\it J. Functional Analysis} {\bf 12}  211--227. 

\bibitem{nourdinviens08} {\sc Nourdin, I.} and {\sc Viens, F.~G. (2008).} Density estimates and concentration inequalities with Malliavin calculus. Available at {\tt http://arxiv.org/pdf/0808.2088v2}

\bibitem{panchenkotalagrand07} {\sc Panchenko, D.} and {\sc Talagrand, M. (2007).} On the overlap in the multiple spherical SK models. {\it Ann. Probab.} {\bf 35} no. 6, 2321--2355.

\bibitem{pickands67} {\sc Pickands, J.,~III (1967).} Maxima of stationary Gaussian processes. {\it Z. Wahrs. und Verw. Gebiete} {\bf 7} 190--223.

\bibitem{pickands69} {\sc Pickands, J.,~III (1969).} 
Upcrossing probabilities for stationary Gaussian processes. {\it Trans. Amer. Math. Soc.} {\bf 145} 51--73.

\bibitem{pickands69a} {\sc Pickands, J.,~III (1969).}  Asymptotic properties of the maximum in a stationary Gaussian process. {\it Trans. Amer. Math. Soc.} {\bf 145} 75--86.

\bibitem{shcherbina97} {\sc Shcherbina, M.~V. (1997).} On the replica-symmetric solution for the Sherrington-Kirkpatrick model. {\it Helv. Phys. Acta } {\bf 70} 838--853.


\bibitem{sheffield07} {\sc Sheffield, S. (2007).} Gaussian free fields for mathematicians. {\it Probab. Theory Related Fields} {\bf  139} no. 3-4, 521--541.


\bibitem{sk75} {\sc Sherrington, D.} and {\sc Kirkpatrick, S. (1975).} Solvable model of a spin glass. {\it Phys. Rev. Lett.} {\bf 35} 1792--1796. 





\bibitem{silveirabouchaud04} {\sc da Silveira, R.~A.} and {\sc Bouchaud, J.-P. (2004).} Temperature and Disorder Chaos in Low Dimensional Directed Paths. {\it Phys. Rev. Lett.,}  {\bf 93} no. 1, 015901.


\bibitem{sinai95} {\sc Sinai, Y.~G. (1995).} A remark concerning random walks with random potentials. {\it Fund. Math.} {\bf 147} no. 2, 173--180.


\bibitem{slepian62} {\sc Slepian, D. (1962).} The one-sided barrier problem for Gaussian noise. {\it Bell System Tech. J.}  {\bf 41} 463--501. 


\bibitem{songzhou96} {\sc Song, R.} and {\sc Zhou, X.~Y. (1996).} A remark on diffusion on directed polymers in random environment. {\it J. Statist. Phys.} {\bf 85} no. 1--2, 277--289.

\bibitem{sudakovtsirelson74} {\sc Sudakov, V.~N.} and  {\sc Cirelson [Tsirelson], B.~S. (1974).} Extremal properties of half-spaces for spherically invariant measures. (Russian) Problems in the theory of probability distributions, II, {\it Zap. Nauchn. Sem. Leningrad. Otdel. Mat. Inst. Steklov. (LOMI)} {\bf 41}, 14--24, 165.

\bibitem{talagrand94} {\sc Talagrand, M. (1994).} On Russo's approximate zero-one law. {\it Ann. Probab.} {\bf 22} 1576--1587. 

\bibitem{talagrand95} {\sc Talagrand, M. (1995).} Concentration of measure and isoperimetric inequalities in product spaces. {\it Inst. Hautes \'Etudes Sci. Publ. Math.} {\bf 81} 73--205. 

\bibitem{talagrand98} {\sc Talagrand, M. (1998).} The Sherrington-Kirkpatrick model: a challenge to math- 
ematicians. {\it Probab. Theory Relat. Fields} {\bf 110} 109--176. 


\bibitem{talagrand03} {\sc Talagrand, M. (2003).}
\newblock {\em Spin glasses: a challenge for mathematicians. Cavity and mean field models}.
\newblock Springer-Verlag, Berlin.

\bibitem{talagrand05} {\sc Talagrand, M. (2005).} {\it The generic chaining. Upper and lower bounds of stochastic processes.} Springer-Verlag, Berlin, 2005.

\bibitem{talagrand06} {\sc Talagrand, M. (2006).} The Parisi formula. {\it Ann. Math. (2)} {\bf 163} 
no 1, 221--263.

\bibitem{tindel05} {\sc Tindel, S. (2005).} On the stochastic calculus method for spins systems. {\it Ann. Probab.} {\bf 33} no. 2, 561--581.


\bibitem{tracywidom94a} {\sc Tracy, C.~A.} and {\sc Widom, H. (1994).} Level-spacing distributions and the Airy kernel. {\it Comm. Math. Phys.} {\bf 159} no. 1, 151--174. 

\bibitem{tracywidom94b} {\sc Tracy, C.~A.} and {\sc  Widom, H. (1994).} Fredholm determinants, differential equations and matrix models. {\it Comm. Math. Phys.} {\bf 163} no. 1, 33--72. 
  
\bibitem{tracywidom96} {\sc Tracy, C.~A.} and {\sc Widom, H. (1996).} On orthogonal and symplectic matrix ensembles. {\it Comm. Math. Phys.} {\bf 177}  no. 3, 727--754.


\bibitem{tis76} {\sc Tsirelson, B.~S., Ibragimov, I.~A.,} and {\sc Sudakov, V.~N. (1976).} Norms of Gaussian sample functions. {\it Proceedings of the Third Japan-USSR Symposium on Probability Theory (Tashkent, 1975),} {\bf 550}, 20--41. 

\bibitem{aizenmanwehr90} {\sc Wehr, J.} and {\sc Aizenman, M. (1990).} Fluctuations of extensive functions of quenched random couplings. {\it J. Statist. Phys.} {\bf 60} no. 3-4, 287--306.


\bibitem{zhang87} {\sc Zhang, Y.-C. (1987).} 
Ground state instability of a random system. {\it Phys. Rev. Lett.,} {\bf 59} no. 19, 2125--2128.


\end{thebibliography}
\end{document}